\documentclass[referee,pdflatex,bst/sn-mathphys]{sn-jnl}

\jyear{2022}%

\theoremstyle{thmstyleone}%
\newtheorem{theorem}{Theorem}
\theoremstyle{thmstyletwo}%

\theoremstyle{thmstylethree}%
\newtheorem{lemma}{Lemma}

\usepackage{cite}
\begin{document}

\title[Asymptotic Analysis of a General Multi-Structured Population Model]{Asymptotic Analysis of a General Multi-Structured Population Model}


\author*[1]{\fnm{Sabina L.} \sur{Altus}}\email{sabina.altus@colorado.edu}

\author[2,3,4]{\fnm{Jeffrey C.} \sur{Cameron}}\email{jeffrey.c.cameron@colorado.edu}

\author*[1]{\fnm{David M.} \sur{Bortz}}\email{david.bortz@colorado.edu}

\affil[1]{\orgdiv{Department of Applied Mathematics}, \orgname{University of Colorado}, \orgaddress{\city{Boulder}, \postcode{80309-0526}, \state{CO}, \country{USA}}}

\affil[2]{\orgdiv{Department of Biochemistry}, \orgname{University of Colorado}, \orgaddress{\city{Boulder}, \postcode{80309-0596}, \state{CO}, \country{USA}}}

\affil[3]{\orgdiv{Renewable and Sustainable Energy Institute}, \orgname{University of Colorado}, \orgaddress{\city{Boulder}, \postcode{80309-0027}, \state{CO}, \country{USA}}}

\affil[4]{\orgname{National Renewable Energy Laboratory}, \orgaddress{\city{Golden}, \postcode{80401-3393}, \state{CO}, \country{USA}}}

\abstract{Structured populations are ubiquitous across the biological sciences.  Mathematical models of these populations allow us to understand how individual physiological traits drive the overall dynamics in aggregate. For example, linear age- or age-and-size-structured models establish constraints on individual growth under which the age- or age-and-size-distribution stabilizes, even as the population continues to grow without bound. However, individuals in real-world populations exhibit far more structural features than simply age and size. Notably, cyanobacteria contain carboxysome organelles which are central to carbon fixation and can be older (if inherited from parent cells) or younger (if created after division) than the enveloping cell. Motivated by a desire to understand how carboxysome age impacts growth at the colony level, we develop a multi-structured model which allows for an arbitrary (but finite) number of structure variables. We then derive and solve the renewal equation for cell division to obtain an asymptotic solution, and show that, under certain conditions, a stable structural distribution is reached.}




\maketitle

\section{Introduction} \label{sec1}
Structured populations are ubiquitous across the biological sciences. Mathematical models of these populations are used in a variety of applications throughout epidemiology \cite{feng2005global, castillo1998global}, ecology \cite{pruss1981equilibrium}, and cell biology \cite{arino1995survey}. The motivating goal is to understand the dynamics of a population for which individual behavior varies according to certain features (e.g., age, size, physiological traits) which may be tracked as structure variables.

In deriving models of this type, we interpret the population as a distribution over its structure variables and impose conservation laws akin to those obeyed in fluid dynamics \cite{MetzDiekmannLevin1986}. As such, structured population models consist of a balance law---a partial differential equation (PDE) describing the evolution of an initial cohort, and a boundary condition describing entry of offspring into the population. The first of this class of mathematical models was the linear age-structured model proposed by Sharpe and Lotka in 1911 (and rederived by McKendrick in 1926) \cite{sharpe1911problem, m1925applications}. Since that time, this model has been generalized many times over to a size-structured model, an age-and-size-structured model, as well as nonlinear versions of all of these \cite{MetzDiekmannLevin1986,sinko1967new,bell1967cell,bellcell} (see G.F. Webb's monograph \cite{webb1985theory} on this topic for a more complete history). In general, analysis of these structured equations has vastly expanded our understanding of how overall population dynamics are governed by the defining features of these classes (i.e., their fecundity and death rates), as well as the competitive or cooperative interactions between them, and their environment. 

A solution to a structured population model determines the evolution of the whole population through careful accounting of individual contributions. The general solution procedure is to apply the method of characteristics, resulting in a piecewise solution which propagates the initial distribution from the $t = 0$ boundary, and future generations from the age $a = 0$ boundary along characteristic curves. Resolving the solution at the $a = 0$ boundary inevitably leads to a Volterra-type integral renewal equation, solved by Laplace transform \cite{bellman1959asymptotic}. The linear multi-structured model presented and analyzed in this work similarly follows this general procedure, but requires an abstract setting. Solutions to the renewal equation are operators in a Banach space, and we use properties of the semigroup associated with the model solution to determine the long-term population behavior through an asymptotic solution.

The multi-structured model was developed with the aim of better understanding how cyanobacterial growth is regulated by the efficiency of their carboxysomes, organelles which facilitate carbon-fixation \cite{hill2020life}. While carboxysome formation and functionality is well established, the relationship between carboxysome age and declining capability has not yet been fully characterized \cite{cameron2013biogenesis, hill2020life}. Cyanobacteria often contain carboxysomes of different ages as they may be inherited from parent cells, or formed after division. By tracking the age of each carboxysome as an additional structure variable, the multi-structured model can help elucidate the relationship to efficiency and impact on growth at both the cellular and population level.

With this application in mind, certain model components are necessarily defined for the cyanobacterial population modeling context, but we aim to maintain generality throughout. At relevant points in this work, we note where researchers could make substitutions for an alternative modeling application.

The paper is structured as follows: In Section \ref{sec2} we present the multi-structured model equations, and specify the necessary components for application to a cell population. In Section \ref{sec3}, we use the method of characteristics to obtain a model solution, and explain why an asymptotic solution will be more informative. In Section \ref{sec4}, we present the semigroup solution and its key properties which validate the analysis presented in Section \ref{sec5}, wherein we derive and solve the abstract renewal equation. We conclude in Section \ref{sec6} with a brief discussion. 

\section{Model Presentation} \label{sec2}
The partial differential equation-based multi-structured model describes the time evolution of a population as the continuous evolution of the distribution of physiological traits across the population. Individual members of the population are fully characterized by their physiological state, that is, their age $a$, and state with respect to a vector of structure variables, $\bar{\mathbf{x}}$, the first of which, $x_1$, refers to size. As such, individuals are expressed as points $(a, \bar{\mathbf{x}}) \in \mathbb{R}^{k + 1}$ in age $a \in \mathbb{R}_+ = [0, \infty)$ and $k$-dimensional state space $\Omega \subset \mathbb{R}^k$, where the state vector $\bar{\mathbf{x}} \in \Omega$ is constructed such that, $$\bar{\mathbf{x}} = \begin{bmatrix} x_1 \\ \mathbf{x} \end{bmatrix} = \begin{bmatrix} x_1 \\ x_2 \\ x_3 \\ \vdots \\ x_k \end{bmatrix} = \begin{bmatrix} \text{Size} \\ \text{Structure variable 1} \\ \text{Structure variable 2} \\ \vdots \\ \text{Structure variable } k-1 \end{bmatrix}.$$ 
The additional structure variables are assumed to take on non-negative values such that $x_i \in \mathbb{R}_+$ for all $i > 1$. Size, $x_1$, however, is restricted to the interval $0<[x_m, x_M]<\infty$, as we impose both a minimum and maximum allowable size, $x_m$ and $x_M$, respectively. It will be necessary at times to refer to age $a$ and size $x_1$ separately from the additional structure variables, hence the notation $\mathbf{x} = [x_2, \cdots, x_k]^T$ in reference to the structure variables apart from size.

Model solutions, $n(t, a, \bar{\mathbf{x}})$, give the age and state distribution of the population at any time $t$, interpreted by integrating over regions of interest in age and state space. For example, the number of individuals at time $t$ of age $a \in [a_1, a_2]$ and state $\bar{\mathbf{x}} \in \omega \subset \Omega$ is, $$\int_{a_1}^{a_2} \int_\omega n(t, a, \bar{\mathbf{x}}) d\bar{\mathbf{x}} da,$$ with $n(t, a, \bar{\mathbf{x}})$ naturally defined on the Banach space $L^1(\mathbb{R}_+ \times \mathbb{R}_+ \times \Omega)$ such that the population size remains finite for finite time.

\subsection{Model Equations}
Consider the multi-structured model equations \begin{subequations} \label{pde_model}
    \begin{align}
        \label{model_eq1} \frac{\partial n}{\partial t} &+ \frac{\partial n}{\partial a} + \nabla \cdot \left[\bar{\mathbf{v}}(a, \bar{\mathbf{x}})n(t, a, \bar{\mathbf{x}})\right] = -\mu(a, \bar{\mathbf{x}})n(t, a, \bar{\mathbf{x}})\\ 
        \label{model_eq3} &n(t, 0, \bar{\mathbf{x}}) = B(t, \bar{\mathbf{x}}) = \int_0^\infty \int_{\Omega_r} \beta(a, \bar{\mathbf{y}}, \bar{\mathbf{x}})n(t, a, \bar{\mathbf{y}})d\bar{\mathbf{y}}da\\
        \label{model_eq2} &n(0, a, \bar{\mathbf{x}}) = \phi(a, \bar{\mathbf{x}})
    \end{align}
\end{subequations}
where $$ \nabla \cdot \left[\bar{\mathbf{v}}(a, \bar{\mathbf{x}})n(t, a, \bar{\mathbf{x}})\right] = \sum_{i=1}^k \frac{\partial \left[v_i(a, \bar{\mathbf{x}})n(t, a, \bar{\mathbf{x}})\right]}{\partial x_i}$$ is the divergence of the vector field $\bar{\mathbf{v}}(a, \bar{\mathbf{x}})n(t, a, \bar{\mathbf{x}})$.
The model consists of three components: the evolution equation $(\ref{model_eq1})$, a linear, hyperbolic partial differential equation expressing the aging and growth processes as translation through state space according to the flow $\bar{\mathbf{v}}(a, \bar{\mathbf{x}})$, the age- and state-specific loss rate $\mu(a, \bar{\mathbf{x}}) = d(a, \bar{\mathbf{x}}) + b(a, \bar{\mathbf{x}})$ due to death $d(a, \bar{\mathbf{x}})$ and division $b(a, \bar{\mathbf{x}})$. The boundary condition $(\ref{model_eq3})$ describes the renewal process, i.e., entry of offspring at the age $a = 0$ boundary. The boundary condition is in the form of a renewal equation as the solution at time $t$ is determined by integrating the birth modulus $\beta(a, \bar{\mathbf{y}}, \bar{\mathbf{x}})$, the average number of offspring of state $\bar{\mathbf{x}}$ produced per unit time by an individual of age $a$ and state $\bar{\mathbf{y}}$, against the population distribution at time $t$. In general, there is a required growth period before individuals may reproduce, so the interval of integration is restricted to the subset $\displaystyle \Omega_r \subset \Omega$ of allowable reproductive states $\bar{\mathbf{y}}$. Lastly, the initial condition $(\ref{model_eq2})$ prescribes the age and state distribution $\phi(a, \bar{\mathbf{x}})$ of the population cohort present at time zero.

\subsection{Application to a Cell Population}
In the modeling context of an evolving cyanobacterial cell population, we can define specific model components through simple, biologically motivated assumptions about cell growth and division. We will use the age of a carboxysome as a proxy for its photosynthetic efficacy, and take each additional structure variable in the vector $\bar{\mathbf{x}}$ to be the age of a single carboxysome. We fix $k$, the number of structure variables in addition to size, to be the number of carboxysomes present at the time of division. That is, a cyanobacteria of age $a = 0$ will have $\displaystyle \frac{k}{2}$ carboxysomes, inherited from its mother, and form an additional $\displaystyle \frac{k}{2}$ carboxysomes before dividing.

\subsubsection{Cell Growth}
Exponential growth is a biologically reasonable assumption for the majority of cell populations, including cyanobacterial \cite{campos2014constant}. Under this assumption, cell size increases in proportion to itself at a constant rate, denoted as $\alpha$. This rate could depend on age and physiological state, as in $\alpha = \alpha(a, \bar{\mathbf{x}})$. 

\subsubsection{Cell Death and Division}
The renewal process of a bacterial cell population is cell division (mitosis). In our model conception, mitosis is an instantaneous event wherein a mother cell of state $\bar{\mathbf{y}}$ divides symmetrically into two identical daughter cells of state $\displaystyle \bar{\mathbf{x}} = \frac{1}{2} \bar{\mathbf{y}}$ appearing at the age $a = 0$ boundary. The symmetric division assumption imposes a partition of the state space into a region $\displaystyle \Omega_r = \left(\frac{x_M}{2}, x_M \right] \times \mathbb{R}_+^{k - 1} \subset \Omega$ of cells large enough to reproduce (divide), and a region $\displaystyle \Omega_b = \left(x_m, \frac{x_M}{2} \right] \times \mathbb{R}_+^{k - 1} \subset \Omega$ of allowable states at birth.

To impose symmetric cell division, we form the birth modulus $\beta$ by applying a Dirac-delta function to each structure variable as in, \begin{equation} \label{birth_modulus}
    \beta(a, \bar{\mathbf{y}}, \bar{\mathbf{x}}) = 2 \beta_1(a) \delta \left(\bar{\mathbf{x}} - \frac{1}{2} \bar{\mathbf{y}} \right) \text{ for } \bar{\mathbf{x}} \in \Omega_b, \ \bar{\mathbf{y}} \in \Omega_r,
\end{equation}
where the factor of two balances the loss of the single mother cell with the appearance of two daughter cells, both of age $a = 0$. The term $\beta_1(a)$ represents the probability of division as a function of age alone. Interpreting the Dirac-delta function appearing in (\ref{birth_modulus}) above as a probability distribution, we see that the probability of a mother cell of state $\bar{\mathbf{y}}$ producing a daughter cell of any state $\bar{\mathbf{x}}$ other than $\frac{1}{2}\bar{\mathbf{y}}$ is zero. 

Integrating the birth modulus over $\Omega_b$ with respect to $\bar{\mathbf{x}}$ gives the total average number of daughter cells produced per unit time by a mother cell of age $a$ and state $\bar{\mathbf{y}}$. This becomes the rate of cell loss due to division, \begin{align} \begin{split}
        b(a, \bar{\mathbf{y}}) &= \frac{1}{2}\int_{\Omega_b} \beta(a, \bar{\mathbf{y}}, \bar{\mathbf{x}}) d\bar{\mathbf{x}}\\
        & = \frac{1}{2}\int_{\Omega_b} 2\beta_1(a) \delta \left(\bar{\mathbf{x}} - \frac{1}{2}\bar{\mathbf{y}} \right) d\bar{\mathbf{x}}\\
        &= \beta_1(a)\chi_{\{\bar{\mathbf{y}} \in \Omega_r \}},
\end{split} \end{align}
where $\displaystyle \chi_{\{\bar{\mathbf{y}} \in \Omega_r \}}$ is the indicator function on $\Omega_r$. The factor of $\displaystyle \frac{1}{2}$ balances the removal of a dividing mother cell with her two daughter cells---as in, the rate at which offspring are produced is twice the rate of cell loss due to division. Assuming a constant death rate, $\displaystyle d(a, \bar{\mathbf{y}}) = \mu_d$, the total rate of cell loss per unit time becomes, \begin{equation}
    \mu(a, \bar{\mathbf{y}}) = \mu_d + \beta_1(a)\chi_{\{\bar{\mathbf{y}} \in \Omega_r \}}.
\end{equation}

Now that we have completely specified the model we will present our asymptotic solution.


\section{Model Solution} \label{sec3}
In this section, we describe our two-step procedure for solving the full multi-structured model. We first apply the method of characteristics to obtain a solution to \eqref{pde_model} which fully describes growth and evolution of a sterile population, and leads to a renewal equation at the age $a = 0$ boundary. We then prove the existence of a unique solution for a non-sterile population exists, and derive its power series representation by the method of successive approximations.

\subsection{Movement in State Space}
Growth or change in physiological state may be interpreted as translation through state space according to the flow determined by the velocity vector $\bar{\mathbf{v}}(a, \bar{\mathbf{x}}) = [v_1(a, \bar{\mathbf{x}}) \ \ v_2(a, \bar{\mathbf{x}}) \ \ \cdots \ \  v_k(a, \bar{\mathbf{x}})]^T$ where $\displaystyle v_i(a, \bar{\mathbf{x}}) = \frac{dx_i}{dt}(a, \bar{\mathbf{x}})$. Age- and state-specific velocity functions $v_i(a, \bar{\mathbf{x}})$ are required to be bounded, continuous, and continuously differentiable with respect to each argument. Additionally, $v_i(a, \bar{\mathbf{x}})$ must be strictly positive on the interior of $\Omega$ such that $\forall (a, \bar{\mathbf{x}}) \in (0, \infty) \times (x_m, x_M) \times (0, \infty)^k$, the infimum of $ \lvert v_i(a, \bar{\mathbf{x}}) \rvert > 0$. This guarantees the existence of a uniquely determined, continuous flow along characteristic curves throughout state space. And finally, $v_i(a, \bar{\mathbf{x}})$ must vanish on the boundary of $\Omega$ (denoted $\partial \Omega$) so that all trajectories beginning at time $t$ with $(a, \bar{\mathbf{x}}) \in \mathbb{R}_+ \times \Omega$, remain in the age- and state-space to ensure that the values of $a$ and $\bar{\mathbf{x}}$ remain where the velocity functions are defined. This restriction ensures that the solutions $n(t, a, \bar{\mathbf{x}})$ stay in $L^1(\mathbb{R}_+ \times \mathbb{R}_+ \times \Omega)$.

Since each velocity term $v_i(a, \bar{\mathbf{x}})$ is required to be continuously differentiable with respect to $x_i$, we note that the balance law (\ref{model_eq1}) may be written as a directional derivative in the $\langle 1, 1, \bar{\mathbf{v}} \rangle$ direction, 
\begin{equation} \label{dirderiv}
    D_{\langle 1, 1, \bar{\mathbf{v}} \rangle}n(t, a, \bar{\mathbf{x}}) = -\bigg(\mu(a, \bar{\mathbf{x}}) + \sum_{i = 1}^k \frac{\partial v_i}{\partial x_i}(a, \bar{\mathbf{x}})\bigg)n(t, a, \bar{\mathbf{x}}).
\end{equation}
This directional derivative (\ref{dirderiv}) gives the instantaneous rate of change at time $t$ in the direction of aging and growth from every age and state in $\mathbb{R}_+ \times \Omega$. 

\subsection{Characteristic and Growth Curves}
Setting the right-hand-side of (\ref{dirderiv}) equal to zero indicates that the rate of change in density $n$ at any point $(t, a, \bar{\mathbf{x}})$ is zero in the $\langle 1, 1, \bar{\mathbf{v}} \rangle$ direction. A parameterized curve advancing from an initial position $(t, a, \bar{\mathbf{x}})$ in the $\langle 1, 1, \bar{\mathbf{v}} \rangle$ direction (along which $n$ is constant) is called a {\it characteristic curve}. Integrating along these curves produces a solution to the model (\ref{pde_model}) where the density at time $t$ is expressed in terms of the initial data $\phi(a, \bar{\mathbf{x}})$ propagated forward in time along these characteristic curves. The characteristic curves\footnote{Capital letters are used to refer to characteristic curves. Generally, these will be presented with two arguments, as in $X(\theta, x)$, however, the characteristics in state space may also depend on age and other structure variables, considered to be fixed, and will be denoted explicitly by $X(\theta, x; a, \mathbf{x})$ only when necessary.} for this system are solutions $T = T(\theta, t), A = A(\theta, a), X_1 = X_1(\theta, x_1),$ and  $\mathbf{X} = \mathbf{X}(\theta, \mathbf{x})$ to the following system of differential equations, parameterized by the auxiliary variable $\theta$, which measures time when $t \leq a$, and age when $t > a$. 
\begin{align}\label{ode_system}
\begin{split}
    &\frac{d}{d\theta}[ T(\theta, t)] = 1, \ \ T(0, t) = t\\
    &\frac{d}{d\theta}[A(\theta, a)] = 1, \ \ A(0, a) = a\\
    &\frac{d}{d\theta}[X_1(\theta, x_1; a, \mathbf{x})] = v_1(A(\theta, a), x_1, \mathbf{X}(\theta, \mathbf{x})), \ \ X_1(0, x_1) = x_1\\
    &\frac{d}{d\theta}[X_i(\theta, x_i; a, x_{j \neq i})] = v_i(A(\theta, a), X_1(\theta, x_1), \mathbf{X}(\theta, \mathbf{x})), \ \  X_i(0, x_i) = x_i
\end{split}
\end{align}
The characteristics along which time passes and age advances are given by, $$T(\theta, t) = \theta + t \text{ and } A(\theta, a) = \theta + a.$$

For the characteristic curves describing growth, \begin{equation*}
    \frac{d}{d\theta}[X_1(\theta, x_1; a, \mathbf{x})] = v_1(A, x_1, \mathbf{X}) \ \Rightarrow \ \int_{x_1}^{X_1} \frac{d \xi}{v_1(A, \xi, \mathbf{X})} = \theta.
\end{equation*}
Let $\displaystyle G(x) = G(x; a, \mathbf{x}) = \int_{x_m}^{x} \frac{d \xi}{v_1(A(\theta, a), \xi, \mathbf{X}(\theta, \mathbf{x}))}$. Then $G(x)$ gives the \textit{time} required for an individual of age $a$ and state $\mathbf{x}$ to grow from the smallest possible size $x_m$ to arbitrary size $x \leq x_M$. An individual of \textit{fixed} size $x_1$ at time $t$, will reach arbitrary size $x \leq x_M$ a time $G(x) - G(x_1)$ later. Continuing from the above, we find $X_1(\theta, x_1) = G^{-1}(\theta + G(x_1))$.

The inverse, $G^{-1}(\theta; x_1) = G^{-1}(\theta; a, x_1, \mathbf{x})$, is guaranteed to exist as long as the physical growth rate $v_1: [x_m, x_M] \to \mathbb{R}_+$ is uniformly continuous and positive on $[x_m, x_M]$. Thus, we call $G^{-1}(\theta; x_1)$ the \textit{growth curve} as it computes the \textit{size} of an individual after a time period of length $\theta$. For instance, an individual of size $x_1$ at time $t_0$ will be of size $G^{-1}(\theta; x_1)$ at time $t_0 + \theta$. 

For each additional structure variable in $\mathbf{x}$, the characteristic curves $X_i(\theta, x_i; a, x_{j \neq i})$ will be similarly expressed through integral equations. From the last differential equation in \eqref{ode_system}, we find, \begin{equation*}
    \int_{x_i}^{X_i} \frac{d \xi}{v_i(A, X_1, \bar{\mathbf{X}}\lvert_{x_i = \xi})} = \theta.
\end{equation*}
Let \begin{equation*}
    F_i(x) = F_i(x; a, \mathbf{x}) = \int_0^x \frac{d \xi}{v_i(A, X_1, \mathbf{x} \lvert_{x_i = \xi})},
\end{equation*} be the \textit{time} required for the $i^{th}$ structure variable to increase from zero to $x$ along the characteristic curve. Then, $\theta = F_i(X_i) - F_i(x_i)$, and $ X_i(\theta, x_i) = F_i^{-1}(\theta + F_i(x_i)).$ 
The inverse functions, $F_i^{-1}(\theta; x_i) = F_i^{-1}(\theta; a, x_1, \mathbf{x})$, are again guaranteed to exist as long as $v_i: \mathbb{R}_+ \to \mathbb{R}_+$ is positive and uniformly continuous on $\mathbb{R}_+$. We note that $F_i^{-1}(\theta; x_i)$ should be interpreted as the value of the $i^{th}$ structure variable after a time period of length $\theta$.

Consider an individual of age $a$ and state $\bar{\mathbf{x}}$. The vector $\bar{\mathbf{X}}(\theta, \bar{\mathbf{x}})$ is the vector of characteristic curves along which the individual advances to its next state. That is, $\bar{\mathbf{X}}(\theta, \bar{\mathbf{x}})$ gives the state of this individual after a time interval of length $\theta$. Helpfully, $\bar{\mathbf{X}}(-a, \bar{\mathbf{x}})$ gives an individual's state-at-birth, which can always be found by traveling backwards along characteristic curves for a time $a$. 

\subsection{The Method of Characteristics}
We obtain a solution $n(t, a, \bar{\mathbf{x}})$ by integrating the total derivative of the density $n$ along characteristic curves, as in, \begin{equation*}
    \frac{d}{d\theta}\left[n(T(\theta, t), A(\theta, a), \bar{\mathbf{X}}(\theta, \bar{\mathbf{x}})) \right] = \frac{\partial n}{\partial t}\frac{dt}{d\theta} + \frac{\partial n}{\partial a}\frac{da}{d\theta} + \sum_{i = 1}^k \frac{\partial n}{\partial x_i}\frac{dx_i}{d\theta}
\end{equation*}
which we recognize from the model equations (\ref{pde_model}) as equivalent to, $$\frac{dn}{d\theta} = -\bigg(\mu(A, \bar{\mathbf{X}}) + \sum_{i = 1}^k \frac{\partial v_i}{\partial x_i}(A, \bar{\mathbf{X}})\bigg)n(T, A, \bar{\mathbf{X}}).$$
Integration with respect to $\theta$ produces the solution, $$n(T, A, \bar{\mathbf{X}}) = C\exp{\bigg[ -\int_0^\theta \mu(A, \bar{\mathbf{X}}) d\theta' - \int_0^\theta \sum_{i = 1}^k \frac{\partial v_i}{\partial x_i}(A, \bar{\mathbf{X}}) d\theta' \bigg]}.$$
To find the constant term $C$, we divide the $at$-plane into two regions along the line $a = t$, as depicted in Figure \ref{fig:atplane}. In the region where $t < a$, the solution acts to propagate the initial distribution $n(0, a, \bar{\mathbf{x}}) = \phi(a, \bar{\mathbf{x}})$ forward in time. In the region where $t \geq a$, the boundary condition $n(t, 0, \bar{\mathbf{x}}) = B(t, \bar{\mathbf{x}})$ determines a distribution entering at the age $a = 0$ boundary that is then propagated forward in the same way.

\begin{figure}
    \centering
    \includegraphics[width=\textwidth]{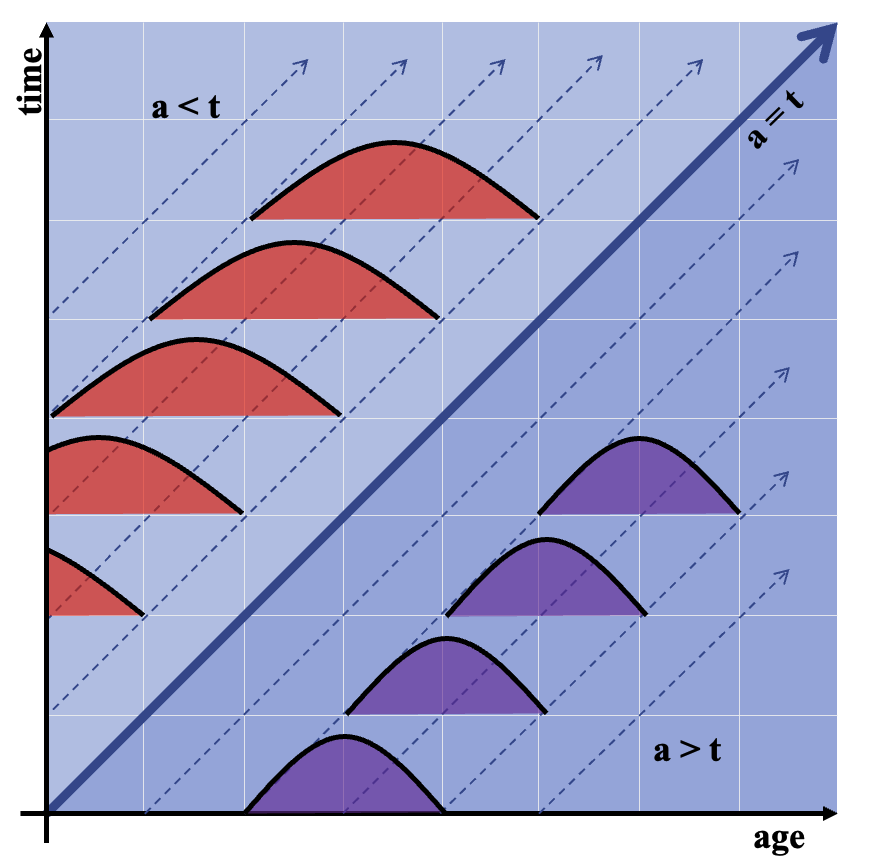}
    \caption{The $at$-plane divided along the line $a = t$. In the darker blue region where $a > t$, the initial age distribution, shown in purple, is propagated forward in time from the $t = 0$ boundary along characteristic curves, shown as dashed blue lines with a slope of 1. In the lighter blue region where $a < t$, the age distribution of offspring, determined from the boundary condition and shown in red, emerges at the age $a = 0$ boundary and propagates forward in time along characteristic curves with slope 1.}
    \label{fig:atplane}
\end{figure}

The {\it survival probability}  $$\Pi(\theta; a, \bar{\mathbf{x}}) = \exp{\left[-\int_0^{\theta} \mu(a - \theta + \sigma, \bar{\mathbf{X}}(\sigma - \theta, \bar{\mathbf{x}}))d\sigma \right]}$$ is the probability that an individual of age $a$ and state $\bar{\mathbf{x}}$ at time $t$ remains in the population at time $t + \theta$; that is, they will not have died or reproduced during the time interval of length $\theta$.  For an age-structured model, the survival probability is entirely sufficient to propagate the density of a population cohort forward in time.

For the multi-structured model, we must also resolve how the volume occupied by a given cohort is distorted as it is translated through state space. For example, imagine a cohort occupying the infinitesimally small size interval $[x_1, x_1 + dx_1]$ initially. Over a time interval of length $d\theta$, the cohort will have expanded to occupy the larger size interval $[x_1 + v_1(a, x_1, \mathbf{x})d\theta, x_1 + dx_1 + v_1(a, x_1 + dx_1, \mathbf{x})d\theta]$. However, from the directional derivative (\ref{dirderiv}), we see that growth as movement through state space (the left-hand-side) must be in balance with volume expansion (the right-hand-side).\footnote{For the moment we are ignoring the loss term.} To eliminate this imbalance, we can identify a correction via computing the Jacobian, \begin{equation*}
    J(\theta; a, \bar{\mathbf{x}}) = \exp\left[-{\int_0^{\theta} \sum_{i=1}^k \frac{\partial v_i}{\partial X_i}d\sigma}\right] = \left\lvert \frac{\partial(A(\theta, a), \bar{\mathbf{X}}(\theta,\bar{\mathbf{x}}))}{\partial (a, \bar{\mathbf{x}})} \right\rvert.
\end{equation*} We note that this is the determinant of the Jacobian matrix of characteristic curves (see Appendix \ref{appendix:jacobian} for proof). 

This term accounts for the fact that a cohort occupying a volume $V$ at time $t$ will grow to occupy a volume that is larger by a factor of $$\exp\left[{\int_0^{\theta} \sum_{i=1}^k \frac{\partial v_i}{\partial X_i}d\sigma}\right]$$ at time $t + \theta$ (for $\theta$ small \cite{bell1967cell}). Said another way, the Jacobian may be interpreted as a coordinate transformation in time from the current time $t$ to a time $t + \theta$ later on, where $\theta$ is small \cite{annosov1997ordinary}. 

Finally, we arrive at the following solution for the population density, 
\begin{equation}\label{n_soln}
    n(t, a, \bar{\mathbf{x}}) = \begin{cases} \phi \left(a - t, \bar{\mathbf{X}}(-t, \bar{\mathbf{x}}) \right) \Pi(t)J(t) \ \ \ \ \ \ \ \text{ for } t < a\\ n \left(t - a, 0, \bar{\mathbf{X}}(-a, \bar{\mathbf{x}}) \right) \Pi(a) J(a) \ \ \text{ for } t > a.
    \end{cases}
\end{equation}
Notice that in the region where $t < a$, the solution is fully determined from the initial condition. 
However, resolving the boundary condition to determine the solution in the $t > a$ region is likely not possible in a closed form. Instead, we derive a series solution below and use it to prove the existence and uniqueness of the asymptotic solution we seek in the following sections. 

\subsubsection{Series Solution}
To obtain the solution to (\ref{pde_model}) where $0 \leq a < t$, the piecewise-defined solution (\ref{n_soln}) for $n$ is inserted into the boundary condition, $$n(t, 0, \bar{\mathbf{x}}) = B(t, \bar{\mathbf{x}}) = \int_0^\infty \int_\Omega \beta(a, \bar{\mathbf{y}}, \bar{\mathbf{x}})n(t, a, \bar{\mathbf{y}})d\bar{\mathbf{y}}da,$$ resulting in the following integral equation for the birth-rate function $B(t, \bar{\mathbf{x}})$,

\begin{align} \begin{split} \label{renewJ}
    B(t, \bar{\mathbf{x}}) = \int_0^t \int_\Omega &\beta(a, \bar{\mathbf{y}}, \bar{\mathbf{x}})B(t - a, \bar{\mathbf{Y}}(-a, \bar{\mathbf{y}}))\\ & \times \exp \left[-\int_0^a \mu(\sigma, \bar{\mathbf{Y}}(\sigma - a, \bar{\mathbf{y}}))d\sigma \right]J(a)d\bar{\mathbf{y}}da\\
    + \int_t^\infty \int_\Omega &\beta(a, \bar{\mathbf{y}}, \bar{\mathbf{x}})\phi(a - t, \bar{\mathbf{Y}}(-t, \bar{\mathbf{y}}))\\ &\times \exp \left[-\int_0^t \mu(a + \sigma -t, \bar{\mathbf{Y}}(\sigma - t, \bar{\mathbf{y}}))d\sigma \right]J(t)d\bar{\mathbf{y}}da\\
   = K(B)(&t, \bar{\mathbf{x}}) + \Phi(t, \bar{\mathbf{x}}),
\end{split}{} \end{align}{}
where $K$ is an operator acting on $B$. We refer to the first integral as $K(B)(t, \bar{\mathbf{x}})$, and the second, $\Phi(t, \bar{\mathbf{x}})$ is defined through the initial condition $\phi$. 

The above equation may be simplified slightly by a change of variables wherein the Jacobian acts as a coordinate transformation on $\Omega$ from a time $dt$ ago to the present. Using the Jacobian matrix as a change of variables from the past, $\bar{\mathbf{Y}}(-t, \bar{\mathbf{y}})$, to the present current-time coordinates, $\bar{\mathbf{Y}}(t, \bar{\mathbf{y}})$, is exactly the inverse of the Jacobian appearing in the solution (\ref{renewJ}) above. The integral equation for $B(t, \bar{\mathbf{x}})$ is then equivalently expressed, 

\begin{align} \begin{split} \label{renewNoJ}
    B(t, \bar{\mathbf{x}}) = \int_0^t \int_\Omega &\beta(a, \bar{\mathbf{Y}}(a, \bar{\mathbf{y}}), \bar{\mathbf{x}})B(t - a, \bar{\mathbf{y}})\\ &\times \exp \left[-\int_0^a \mu(\sigma, \bar{\mathbf{Y}}(\sigma, \bar{\mathbf{y}}))d\sigma \right] d\bar{\mathbf{y}}da\\
    + \int_0^\infty \int_\Omega &\beta(a + t, \bar{\mathbf{Y}}(t, \bar{\mathbf{y}}), \bar{\mathbf{x}})\phi(a, \bar{\mathbf{y}})\\ &\times \exp \left[-\int_0^t \mu(a + \sigma, \bar{\mathbf{Y}}(\sigma, \bar{\mathbf{y}}))d\sigma \right] d\bar{\mathbf{y}}da\\
    = K(B)(&t, \bar{\mathbf{x}}) + \Phi(t, \bar{\mathbf{x}}).
\end{split}{} \end{align}

The following existence and uniqueness theorem adapted from \cite{tucker1988nonlinear} shows that the integral equation (\ref{renewNoJ}) for $B$ has a unique solution $B(t, \cdot)$ which admits a continuous mapping for any $t$ in the finite interval $[0, T]$ to the state space $\Omega$, and that this mapping can be extended to the full space as $t \to \infty$. Though we may not find a closed form solution for $B$, the method of successive approximations gives a solution in the form of a convergent series of repeated applications of the operator $K$ to the initial cohort which defines $\Phi$. In this way, each new application of $K$ corresponds to the next generation.
\begin{theorem}\label{Btheorem}
    There exists a unique, continuous and bounded solution $B$ to (\ref{renewNoJ}).
\end{theorem}
    
\begin{proof}
    Let $\Omega_1 \subset \Omega$ be the volume of state space occupied by the initial cohort, and $\Tilde{\beta}$ an upper bound on the birth modulus $\beta$. (Refer to the integral equations as presented in (\ref{renewJ}) to see clearly how the Jacobian increases $\Omega_1$ with each application of $K$.) The operator $K$ is continuous in $t$ and $\bar{\mathbf{x}}$, and its supremum norm is bounded by, $$\lvert \lvert K \rvert \rvert \leq t \Tilde{\beta} \Omega_1 \sup_{0 \leq \sigma \leq t} \lvert \lvert B \rvert \rvert.$$ Similarly, $\Phi$ is continuous in $t$ and $\bar{\mathbf{x}}$, and bounded by, $\lvert \lvert \Phi \rvert \rvert \leq \Tilde{\beta} \lvert \lvert \phi \rvert \rvert_{L^1}$. By the method of successive approximations, $$B = \Phi + K(\Phi) + K^2(\Phi) + \cdots = \sum_{N = 0}^\infty K^N(\Phi).$$ The series solution for $B$ converges for values of $t$ in the finite interval $[0, T]$ where $T \leq \frac{1}{\Tilde{\beta} \Omega_1}$. We require $T \leq \frac{1}{\Tilde{\beta} \Omega_1}$ to ensure that repeated applications of $K$, each representing the next newly born cohort, remain bounded as $$\lvert \lvert K^N(\Phi) \rvert \rvert \leq (T \Tilde{\beta} \Omega_1)^N \times \Tilde{\beta} \lvert \lvert \phi \rvert \rvert_{L^1}.$$ Since $B$ is the uniform limit of continuous functions, it is also continuous.\\ Assuming there are two solutions, $B_1$ and $B_2$, and inserting their difference $B_1 - B_2$ into the above inequality in place of $\Phi$ shows that the difference must be zero, and therefore the solution must be unique as well.
\end{proof}

We have shown that a unique and continuous solution $B$ to the integral equation \begin{equation} \label{Beqn}
    B(t, \bar{\mathbf{x}}) = K(B)(t, \bar{\mathbf{x}}) + \Phi(t, \bar{\mathbf{x}}),
\end{equation} 
exists, and can by approximated by a power series. The series converges to the solution $B$ on a closed time interval, however, the length of the interval grows with the addition of each new term, or generation. In other words, we can always find the distribution of newborn cells as the sum of the contribution, determined through $K$, from all of the previous cohorts, and the population will become infinitely large in infinite time. Nevertheless, the guaranteed existence and uniqueness of a solution $B$ allows us to extract valuable information about the long-term behaviour of the population---particularly, if a stable distribution is reached in age and among the structure variables.

An asymptotic solution describes the behavior of a system as time increases to infinity. Generally, there is a short, transient phase before the promised asymptotic behavior is realized, the challenge is in separating out the dominant behavior that will persist over time and showing that all other contributions quickly become negligible. The asymptotic solution is analogous to an equilibrium solution for a linear model in that, the population will continue to increase for all time (or go extinct) while the relative fraction of the total population in a given state remains constant.

If such a solution exists, evolution in time can be separated from the structural distribution yielding solutions $B(t, \bar{\mathbf{x}})$ of the form $$B(t, \bar{\mathbf{x}}) = e^{\lambda t} \psi(a, \bar{\mathbf{x}}).$$ In the following section, we will show that linear operator associated with the PDE (\ref{pde_model}) is the generator of a strongly continuous semigroup, and that the spectral properties of this operator determine conditions for a steady-state solution. This fully justifies our casting of (\ref{renewJ}) as an abstract renewal equation in Section \ref{sec5}.

\section{Model Associated Linear Operator and Semigroup} \label{sec4}
Resolving the solution at the age $a = 0$ boundary requires that solutions are guaranteed to be bounded in a specific way. Here we will show that the solution (\ref{n_soln}) generates a semigroup of linear operators, and introduce properties of the semigroup which guarantee well-posedness of the multi-structured model (\ref{pde_model}), as well the necessary boundedness.

\subsection{The Abstract Cauchy Problem}
Let $U$ be the Banach space $L^1(\mathbb{R}_+ \times \Omega)$. The time evolution of $n(t, a, \bar{\mathbf{x}})$ is described by a function mapping $t \in \mathbb{R}_+ \to n(t, \cdot, \cdot) \in U$, governed by the ACP, \begin{equation}\tag{ACP} \label{ACP}
    \frac{dn(t)}{dt} = \mathcal{A}n(t), \ \ n(0) = \phi
\end{equation} where $\mathcal{A}$ is a linear, closed, and generally unbounded (differential) operator with dense domain $\mathcal{D}(\mathcal{A}) \in U$ \cite{pazy2012semigroups}. The word \textit{abstract} in Abstract Cauchy Problem signifies that solutions are Banach space valued.

A function $n: \mathbb{R}_+ \to U$ is a solution of (\ref{ACP}) if it is continuously differentiable, takes on values in $\mathcal{D}(\mathcal{A})$, and satisfies the ACP. The ACP is \textit{well-posed} if for every initial state $\phi \in \mathcal{D}(\mathcal{A})$, there exists a unique solution with continuous dependence on $\phi$. Solutions of a well-posed ACP give rise to a family $\{S(t)\}$ of bounded linear operators on $U$, defined as the unique set of operators satisfying $n(t) = S(t)\phi$ \cite{greiner1988growth}. Finally, this family of operators $\{S(t)\}$ is a \textit{strongly-continuous semigroup}, meaning, it satisfies the following four defining properties: 
\begin{enumerate}
    \item $S(t)$ is a continuous mapping from $U$ into itself.
    \item $S(0) = I.$
    \item The semigroup property, $S(s)S(t)\phi = S(t+s)\phi$.
    \item Strong continuity, $\lim_{t \searrow 0} \lvert \lvert S(t)\phi - \phi\lvert \lvert  = 0.$
\end{enumerate}{}

On the other hand, each semigroup can be associated with a closed, densely defined operator $\mathcal{A}$ called the \textit{infinitesimal generator}, or simply the \textit{generator}, of $S$, defined \begin{equation}
    \mathcal{A}\phi = \lim_{t \searrow 0}\frac{1}{t}(S(t)\phi - \phi).
\end{equation}
The operator $\mathcal{A}$ uniquely determines the semigroup, and gives rise to a well-posed ACP \cite{pazy2012semigroups}. The relationship between the ACP, a semigroup, and its generator is summarized by the Well-Posedness Theorem \cite{neumann2002evolution, greiner1988growth}:

\begin{theorem}{Well-Posedness Theorem.}
    For a closed linear operator $\mathcal{A}$ with domain $\mathcal{D}(\mathcal{A})$ dense in a Banach space $U$, the following properties are equivalent: \begin{enumerate}
       \item The ACP defined on $U$ is well-posed.
       \item The operator $\mathcal{A}$ is the generator of a strongly-continuous semigroup $\{S(t)\}_{t \geq 0}$ on $U$, and classical solutions of the ACP are given by $n(t) = S(t)\phi$ for $\phi \in \mathcal{D}(\mathcal{A})$.
   \end{enumerate}{} 
\end{theorem}

\subsection{Spectral Properties of the Generator}
While the ultimate goal is to solve (\ref{ACP}), it is the spectral properties of the generator $\mathcal{A}$ that give us an idea of the asymptotic behavior of solutions.  This is made clear by the Hille-Yosida Theorem, the central theorem of semigroup theory which distiguishes the generators of strongly continuous semigroups among the class of all linear operators. Before presenting the theorem as it appears in \cite{pazy2012semigroups}, we first state the following lemma and prove the claim that every strongly continuous semigroup is \textit{exponentially bounded} (which follows from the third semigroup property listed above). 

\begin{lemma}
    Let $S(t)$ be a strongly continuous semigroup, and let $\omega \in \mathbb{R}$, and $M \geq 1$ be constants. Then $S(t)$ satisfies $\lvert \lvert S(t)\lvert \lvert  \leq Me^{\omega t}$ for all $t \geq 0$.
    \begin{proof}
        Choose $M \geq 1$ such that $\lvert \lvert S(s)\lvert \lvert  \leq M$ for all $0 \leq s \leq 1$ and write $t \geq 0$ as $t = s + n$ for $n \in \mathbb{N}$ and $0 \leq s < 1$. Then, $$\lvert \lvert S(t)\lvert \lvert \ \leq \ \lvert \lvert S(s)\lvert \lvert \cdot\lvert \lvert S(1)\lvert \lvert ^n \ \leq \ M^{n+1} = Me^{n\ln(M)} \ \leq \ Me^{\omega t}$$ holds for $\omega := \ln(M)$ and each $t \geq 0$  \cite{engel2006short}.
    \end{proof}
\end{lemma}

\begin{theorem}[Hille-Yosida]
    Let $\mathcal{A}$ be a closed, and densely defined linear operator on a Banach space $U$. $\mathcal{A}$ is the generator of a strongly continuous semigroup $S(t)$ if and only if the half-line $(\omega, \infty)$ is contained in the resolvent set $\rho(\mathcal{A})$, and $$\lvert \lvert R(\lambda, \mathcal{A})^m\lvert \lvert  \leq \frac{M}{(Re(\lambda) - \omega)^m}, \ \ \forall \ \lambda > \omega, \ m \in \mathbb{N}.$$
\end{theorem}

From the Hille-Yosida Theorem, we find that spectral properties of $\mathcal{A}$ signify the existence of a corresponding semigroup $S$, and therefore the associated ACP is well-posed \cite{van2012asymptotic}. The question is, which spectral properties of $\mathcal{A}$ will allow us to make conclusions about the asymptotic behavior of $S$, and thus the solutions of the ACP. Of particular interest are the \textit{spectrum of an operator}, the \textit{resolvent set}, the \textit{resolvent operator}, and the \textit{spectral radius} defined in Appendix \ref{appendix:definitions}.

The primary result from the spectral properties of the generator, $\mathcal{A}$, is that it allows us to establish a growth bound on the semigroup it generates, and therefore on solutions to the associated ACP. The uniform growth bound, $\omega_0(S)$, is defined as $$\omega_0(S) := \inf \{\omega \in \mathbb{R} : \exists M > 0 \text{ such that } \lvert \lvert S(t)\lvert \lvert  \leq Me^{\omega t}, \forall t \geq 0 \}.$$ From the Hille-Yosida Theorem, we have that the spectrum of the generator of a strongly continuous semigroup is always contained in some left half-plane, as in, the maximum real part of an element of the spectrum defines an infinite vertical boundary (an abscissa) and all other elements in the spectrum of the generator are contained in the half-plane to the left of this boundary. We can then define the spectral bound $s(\mathcal{A})$ by $$s(\mathcal{A}) := \sup \{Re(\lambda) : \lambda \in \sigma(\mathcal{A}) \}.$$ From the Perron-Frobenius Theorem for positive semigroups, we have that for a positive semigroup, $s(\mathcal{A})$ is always in the spectrum of $\mathcal{A}$ \cite{greiner1988growth}. 

The Spectral Mapping Theorem states that, for a linear operator $L$, and an analytic function $f$, $$\sigma(f(L)) = f(\sigma(L)).$$ The Taylor series for an exponential function is everywhere convergent, so that $\sigma(e^{t\mathcal{A}})$ is equal to $e^{t \sigma(\mathcal{A})}$. Therefore, $e^{t (\omega_0(S))} = r(e^{t\mathcal{A}}) = e^{t s(\mathcal{A})}$, from which we can conclude that $s(\mathcal{A}) \leq \omega_0(S)$, with equality if and only if $\mathcal{A}$ is bounded \cite{van2012asymptotic}.



\subsection{Semigroup Solution and ACP for the Multi-Structured Model}
The solution given in (\ref{n_soln}) forms a strongly-continuous semigroup of linear operators $\{S(t)\}_{t \geq 0}$ where, \begin{equation}\label{semigroupsolution}
    n(t, a, \bar{\mathbf{x}}) = (S(t)\phi)(a, \bar{\mathbf{x}}) = \begin{cases} \phi(a - t, \bar{\mathbf{X}}(-t, \bar{\mathbf{x}})) \Pi(t)J(t) \ \ \ \ \ \ \ \text{ for } t < a\\ n(t - a, 0, \bar{\mathbf{X}}(-a, \bar{\mathbf{x}})) \Pi(a) J(a) \ \ \text{ for } t > a.
    \end{cases}
\end{equation}

\begin{lemma}
    $\{S(t)\}$ is a strongly continuous semigroup on $U$.
\end{lemma}

\begin{proof}
    Provided in Appendix \ref{appendix:semgiroupprops}.
\end{proof}

While $\{S(t)\}_{t \geq 0}$ has several convenient properties by virtue of being a strongly-continuous semigroup, such as exponential boundedness, in order to make use of them, we must show that $\{S(t)\}_{t \geq 0}$ is \textit{generated} by the linear operator associated with the model (\ref{pde_model}). The linear operator $\mathcal{A}$ associated with the PDE in (\ref{pde_model}) is, 
\begin{equation}\label{aphi}
    \mathcal{A}\phi = -\left[\frac{\partial \phi}{\partial a} + \sum_{i=1}^k v_i \frac{\partial \phi}{\partial x_i} + \phi \sum_{i=1}^k \frac{\partial v_i}{\partial x_i} + \mu \phi \right].
\end{equation}

Through this operator, the PDE model may be recast as the Abstract Cauchy Problem in the Banach space $U = L^1(\mathbb{R}_+ \times \Omega)$, for $n(t, \cdot, \cdot)$: \begin{equation} \label{ACPformodel}
        \frac{dn(t)}{dt} = \mathcal{A}n(t), \ \ n(0) = \phi.
\end{equation}

Furthermore, we have proven (in Appendix \ref{appendix:generatorproof}) that $\mathcal{A}$ is the infinitesimal generator of $\{S(t)\}_{t \geq 0}$ by showing that $$\mathcal{A}\phi = \lim_{t \searrow 0} \frac{1}{t} (S(t)\phi - \phi)$$ is satisfied for every $\phi$ in the domain of $\mathcal{A}$, $\mathcal{D}(\mathcal{A})$ \cite{engel2001one}. The Well-Posedness Theorem then guarantees that the ACP (\ref{ACPformodel}) is well-posed, and therefore $n(t, \cdot, \cdot) = S(t)\phi$ is the unique classical solution for $\phi \in \mathcal{D}(\mathcal{A})$.

\subsubsection{The Domain of the Generator}
The definition of the semigroup $S(t)$ and its generator $\mathcal{A}$ are incomplete without a definition of the domain, $\mathcal{D}(\mathcal{A})$. In the most general sense, the domain of a semigroup is a Banach space, or subset of a Banach space, where its generator is defined \cite{engel2006short}. For this problem, $\mathcal{D}(\mathcal{A})$, is the subset of all $\phi$ in the Banach space $L^1(\mathbb{R}_+ \times \Omega)$ satisfying the following two properties: \begin{enumerate}
        \item The boundary conditions must match. This means $$\lim_{t \to 0^+} \frac{1}{t} \int_0^t \int_{\Omega} \lvert \phi(a, \bar{\mathbf{x}}) - \hat{B}(0, \bar{\mathbf{x}}) \lvert d\bar{\mathbf{x}}da = 0$$ where $\displaystyle \hat{B}(t, \bar{\mathbf{x}}) = \int_0^\infty \int_{\Omega} \beta(a, \bar{\mathbf{y}}, \bar{\mathbf{x}})S(t)u(a, \bar{\mathbf{y}})d\bar{\mathbf{y}}da$.
        \item The derivative must remain in the domain. This means that there exists some $\phi' \in L^1(\mathbb{R}_+ \times \Omega)$ such that for every $\tau \in \mathbb{R}$, $$\phi(a - \tau, \bar{\mathbf{G}}^{-1}(\bar{\mathbf{G}}(\bar{\mathbf{x}}) - \tau)) - \phi(a, \bar{\mathbf{x}}) = \int_0^{\tau} \phi'(a - s, \bar{\mathbf{G}}^{-1}(\bar{\mathbf{G}}(\bar{\mathbf{x}}) - s))ds$$ where $\phi(a - \tau, \bar{\mathbf{G}}^{-1}(\bar{\mathbf{G}}(\bar{\mathbf{x}}) - \tau))$ is defined to be zero on $[0, -\tau) \times \Omega$ if $\tau < 0$. If this is satisfied, then the directional derivative $$D_{(1, \bar{\mathbf{v}})}\phi(a, \bar{\mathbf{x}}) = \lim_{\tau \to 0} \frac{1}{\tau} [\phi(a - \tau, \bar{\mathbf{G}}^{-1}(\bar{\mathbf{G}}(\bar{\mathbf{x}}) - \tau)) - \phi(a, \bar{\mathbf{x}})]$$ exists almost everywhere and is equal to $\phi'(a, \bar{\mathbf{x}})$
    \end{enumerate}{}
Under these conditions, we can restrict our domain to the Sobolev space $W^{1, 1}(\mathbb{R}_+ \times \Omega)$, giving us the correct balance of regularity and integrability \cite{hunter2001applied}.

\section{Abstract Renewal Equation} \label{sec5}
In this section, we seek an asymptotic solution to the model (\ref{pde_model}) by deriving and analyzing an abstract renewal equation based on integral equation (\ref{Beqn}) at the age $a = 0$ boundary. We have shown in the previous section that the model solution generates a strongly-continuous semigroup of linear operators, and therefore, we are justified in seeking a solution which maps points in time to locations in state space. Furthermore, the semigroup guarantees our solutions are exponentially bounded, which will guarantee convergence of the Laplace integrals we use to solve the renewal equation in its abstract form. In our analysis, the abstract renewal equation is reduced to an eigenproblem wherein time evolution is separated from the structural distribution yielding solutions of the form $$B(t, \bar{\mathbf{x}}) = e^{\lambda t} \psi(a, \bar{\mathbf{x}}),$$ where the constant $\lambda$ is the Malthusian growth parameter, intrinsic to the population itself. The following analysis is largely based off three fundamental sources: Heijmans \cite{MetzDiekmannLevin1986}, Tucker and Zimmerman \cite{tucker1988nonlinear}, and Bell and Anderson \cite{bell1967cell}.

\subsection{Reduction to an Abstract Renewal Equation}
Beginning with the integral equation (\ref{renewJ}) for $B(t, \bar{\mathbf{x}})$, we identify $\Phi(t, \bar{\mathbf{x}})$ and $B(t, \bar{\mathbf{x}})$ with their respective mappings $t \to \Phi(t, \cdot)$ and $t \to B(t, \cdot)$. Thus, we can express (\ref{renewJ}) as the abstract renewal equation, \begin{equation} \label{abstrRR}
    B(t) = \Phi(t) + \int_0^t K(a) B(t - a) da.
\end{equation}
where, for each \textit{fixed} $t$, $B(t)$ and $\Phi(t)$ are operators in the space $L_1(\Omega)$ in that $B(t, \bar{\mathbf{x}}) = (B(t))(\bar{\mathbf{x}})$, and similarly for $\Phi$. 

For each \textit{fixed} $t$, $K(t)$ defines a bounded linear operator from $L_1(\Omega) \to L_1(\Omega)$. Let $\psi(\bar{\mathbf{x}})$ be an arbitrary $L_1$-function on $\Omega$. Then, $(K(t) \psi)(\bar{\mathbf{x}})$ maps $\psi(\bar{\mathbf{x}}) \in L_1(\Omega)$ to another function in $L_1(\Omega)$ through multiplication and translation by the kernel $\mathcal{K}(t, \bar{\mathbf{x}})$. Inserting the birth modulus defined in (\ref{birth_modulus}), we can express the operator $K$ as, \begin{align} \begin{split}
    K(B)(t, \bar{\mathbf{x}}) = \int_0^t \int_\Omega &\beta(a, \bar{\mathbf{y}}, \bar{\mathbf{x}})B\left(t - a, \bar{\mathbf{Y}}(-a, \bar{\mathbf{y}})\right)\\ &\times \Pi(a; a, \bar{\mathbf{x}}) J(a; a, \bar{\mathbf{x}}) d \bar{\mathbf{y}} da\\
    = \int_0^t \int_\Omega &\beta_1(a) \chi_{\{\bar{\mathbf{y}} \in \Omega_r \}} \delta\left(\bar{\mathbf{x}} - \frac{1}{2}\bar{\mathbf{y}}\right) B\left(t - a, \bar{\mathbf{Y}}(-a, \bar{\mathbf{y}})\right)\\ &\times \Pi(a; a, \bar{\mathbf{x}}) J(a; a, \bar{\mathbf{x}}) d \bar{\mathbf{y}} da\\
    = \int_0^t &\beta_1(a) \chi_{\{\bar{\mathbf{y}} \in \Omega_r \}} B\left(t - a, \bar{\mathbf{Y}}(-a, 2\bar{\mathbf{y}})\right)\\ &\times \Pi(a; a, 2\bar{\mathbf{x}}) J(a; a, 2\bar{\mathbf{x}}) da\\
    = \int_0^t &\mathcal{K}(a, 2\bar{\mathbf{x}}) B(t - a, \bar{\mathbf{Y}}(-a, 2\bar{\mathbf{y}})) da.
\end{split}
\end{align}
We define $\mathcal{K}(a, \bar{\mathbf{x}})$ as $\displaystyle \mathcal{K}(a, \bar{\mathbf{x}}) :=  \beta_1(a) \Pi(a; a, \bar{\mathbf{x}}) J(a; a, \bar{\mathbf{x}}) \chi_{\{\bar{\mathbf{y}} \in \Omega_r \}}$, and see that, $$(K(t) \Psi)(\bar{\mathbf{x}}) = \mathcal{K}(t, 2\bar{\mathbf{x}})\Psi(\bar{\mathbf{X}}(-t,  2\bar{\mathbf{x}})) \in L_1(\Omega)$$ whenever $\Psi(\bar{\mathbf{X}}) \in L_1(\Omega)$.

\subsubsection{Domain and Range of Operators}
Here we present and discuss the continuity requirements and restrictions for each operator. The renewal equation is composed of three operators, $B(t, \bar{\mathbf{x}})$, $\Phi(t, \bar{\mathbf{x}})$, and $\mathcal{K}(a, \bar{\mathbf{x}})$, which functions as an integral kernel to be integrated against $B(a, \bar{\mathbf{x}})$ with respect to $a$ on $[0, t]$. Each of these continuous functions is non-negative, has compact support, and acts on the space $\mathbb{R}_+ \times \Omega$. 

\begin{lemma}
    For each $t \in [0, T]$, there is a unique, continuous mapping from the interval $[0, T]$ into $L_1(\Omega)$, such that $\Phi(t, \cdot)$ and $B(t, \cdot)$ satisfy the integral equation (\ref{renewJ}).
\end{lemma}

\begin{proof}
    $\Phi(t, \bar{\mathbf{x}})$ is obtained by integrating the initial condition $\phi(a, \bar{\mathbf{x}})$ along characteristic curves for $a \in (0, \infty)$. Since $\phi$ is assumed to be a Lipschitz continuous, compactly supported function in $L_1(\mathbb{R}_+ \times \Omega)$, integration over all ages $a$ produces a continuous function in $L_1(\Omega)$. In this way, for each fixed $t \in [0, T]$, there exists a unique, continuous mapping $t \to \Phi(t, \bar{\mathbf{x}})$ from the interval $[0, T] \to L_1(\Omega)$, with $\Phi(t, \cdot)$ satisfying the integral equation (\ref{renewJ}).\\ The result for $B$ was proven in Theorem \ref{Btheorem}.
\end{proof}

Further, since $G(2x_1)$ is the time required for the smallest possible cell to reach a length of $2x_1$, we consider $\Phi(t, \bar{\mathbf{x}}) = 0$ if $t \geq G(2x_1)$, because any cell present at time $t = 0$ will necessarily have surpassed this length, and there will be no contribution of new cells of length $x_1$.

The integration kernel $\mathcal{K}(a, \bar{\mathbf{x}})$ is continuous in both $a$ and $\bar{\mathbf{x}}$. Continuity of $\mathcal{K}$ again guarantees the existence of a unique, continuous $L_1$-function $B$ that is the solution to (\ref{renewJ}). $\mathcal{K}(a, 2\bar{\mathbf{x}}) = 0$ if $\bar{\mathbf{x}} \not\in \Omega_r$, as these cells are too small to divide, and $\mathcal{K}(a, 2\bar{\mathbf{x}}) = 0$ if $a \geq G(2x_1)$ as, similar to $\Phi$, these cells are older than the time it takes to grow to size $2x_1$, and would therefore be too large to produce cells of size $x_1$.

\subsection{Associated Eigenproblem Derivation and Solution}
As is standard practice when solving renewal-type integral equations, we take the Laplace Transform of the abstract renewal equation and find that the operator $B(t)$ can be expressed as an inverse Laplace Transform. Analysis of this integral leads to an eigenproblem from which we arrive at the asymptotic solution. 

\subsubsection{Laplace Transform}
Taking the Laplace transform of both sides of the abstract renewal equation, \begin{equation*} \label{renewalcombo}
    \mathcal{L}\{B(t)\} = \mathcal{L}\left\{\Phi(t) + \int_0^t K(a)B(t-a)da \right \}
\end{equation*}{} yields, \begin{equation*}
    \hat{B}(\lambda) = \hat{\Phi}(\lambda) + \hat{K}(\lambda)\hat{B}(\lambda),
\end{equation*}{}
from which we find, $$\hat{B}(\lambda) = \left(I - \hat{K}(\lambda) \right)^{-1}\hat{\Phi}(\lambda)$$
The solution $B(t)$ is then the inverse Laplace transform of $\hat{B}(\lambda)$,
$$B(t) = \mathcal{L}^{-1}\{\hat{B}(\lambda)\} = \frac{1}{2 \pi i} \int_{c - i \infty}^{c + i \infty} e^{\lambda t}(I - \hat{K}(\lambda))^{-1}\hat{\Phi}(\lambda)d\lambda.$$
Cauchy's Residue Theorem says that the solution to this complex integral is the sum over the residues of $\hat{B}(\lambda)$ at each pole. Using properties of the semigroup solution $S(t)$, we will show that singularities of $\hat{B}(\lambda)$ only exist when $(I - \hat{K}(\lambda))^{-1}$ is singular, and that this a simple pole defined for a unique value of $\lambda$. 

\subsubsection{Derivation of the Eigenproblem}
We have shown previously that the solution to the model (\ref{pde_model}) generates a semigroup of linear operators. Crucially, the semigroup solution is exponentially bounded, and therefore its constituent operators $K(t)$ and $\Phi(t)$ are exponentially bounded as well. From the survival probability, we express the operator bound through a bound on $\mu$, the probability of cell loss. 

Let $$ \mu_\infty = \lim_{\sigma \to \infty} \mu(a + \sigma, \bar{\mathbf{X}}(\sigma, \bar{\mathbf{x}})) \leq d + 1 < \infty.$$ Then, $\displaystyle \lvert \lvert \Phi(t)\lvert \lvert  \leq M_1e^{-\mu_\infty t}$, and $\displaystyle \lvert \lvert K(t) \psi\lvert \lvert  \leq \lvert \lvert \psi\lvert \lvert  M_2e^{-\mu_\infty t}$. Therefore, $\hat{K}(\lambda)$ and $\hat{\Phi}(\lambda)$ are both analytic where they are defined, that is for all $\lambda$ in the right-half plane $$\Lambda := \{\lambda \in \mathbb{C} \lvert Re(\lambda) > - \mu_\infty \}.$$ The operators are not defined for $Re(\lambda) < -\mu_\infty $ as the exponent in the Laplace transform would become positive forcing the integral to diverge. 

As the operators $\hat{K}(\lambda)$ and $\hat{\Phi}(\lambda)$ are both analytic in $\Lambda$, the only singularities will arise when $1 \in \sigma(\hat{K}(\lambda))$, the spectrum of the Laplace transform of $K(t)$. Therefore, the long-term behavior of $B(t)$ will be determined by the element $\lambda \in \Lambda$ with the largest real part such that $I - \hat{K}(\lambda)$ is singular.

In what follows, we present a series of lemmas as laid out in Heijman's analysis of an age- and size-structured model appearing with proofs in \cite{MetzDiekmannLevin1986}. These results hold for the multi-structured model and together assert that there is one dominant eigenvalue $\lambda_0$, equal to the spectral radius of the semigroup $S(t)$, thus determining the long-term behavior of the system along with the corresponding eigenfunction $\psi_{\lambda_0}$.
\begin{lemma}
    For all $\lambda \in \Lambda$, $\hat{K}(\lambda)$ is compact.
\end{lemma}
Establishing that $\hat{K}(\lambda)$ is both analytic and compact for $\lambda \in \Lambda$ leads to an important conclusion about the inverse of $(I - \hat{K})$. From these two properties, we can show (via the Open Mapping Theorem) that $(I - \hat{K})$ is either nowhere invertible, or it has an inverse with a special property, and we call the inverse $(I - \hat{K})^{-1}$ \textit{meromorphic}. We know that $(I - \hat{K})$ is at least invertible for large values of $\lambda$ from the definition of the Laplace Transform\footnote{Since $K(t)$ is bounded, its Laplace Transform $\displaystyle \hat{K}(\lambda) = \int_0^\infty e^{-\lambda t}K(t)dt$ is also bounded for $\lambda \in \Lambda$. Taking the limit as $\lambda$ goes to infinity, we find that $\displaystyle \lim_{\lambda \to \infty} \lvert \lvert \hat{K}(\lambda) \rvert \rvert = \lim_{\lambda \to \infty} \lvert \lvert \int_0^\infty e^{-\lambda t}K(t)dt \rvert \rvert \leq \lim_{\lambda \to \infty} \int_0^\infty e^{-\lambda t} \lvert \lvert K(t) \rvert \rvert dt = 0.$ Therefore, eigenvalues of $(I - \hat{K})$ are bounded away from zero, and the operator is invertible, when $\lambda$ is large enough.}, justifying the claim of the following lemma. 

\begin{lemma}
    The function $\lambda \to (I - \hat{K}(\lambda))^{-1}$ is meromorphic in $\Lambda$.
\end{lemma}
If the mapping from $\lambda$ to the operator $(I - \hat{K}(\lambda))^{-1}$ is meromorphic, the set $$\Sigma = \{\lambda \in \Lambda \lvert 1 \in \sigma(\hat{K}(\lambda))\}$$ is a discrete set whose elements are poles of $(I - \hat{K}(\lambda))^{-1}$ of finite order. 
\begin{lemma}
    If $\psi$ is an eigenfunction of $\hat{K}(\lambda)$, then $\psi(\bar{\mathbf{x}}) = 0$ for $\bar{\mathbf{x}} \not\in \Omega_r$.
\end{lemma}
From this lemma, we see that repeated applications of $\hat{K}(\lambda)$ are ultimately restricted to the subspace of $L_1(\Omega)$ spanned by $\psi$, and therefore, $\hat{K}(\lambda)$ restricted to this subspace is \textit{non-supporting}. If $\hat{K}(\lambda)$ is non-supporting, then the spectral radius $r = r(\hat{K}(\lambda))$, is a pole of the resolvent, $(\lambda I - \hat{K}(\lambda))^{-1}$, and an algebraically simple eigenvalue of $\hat{K}(\lambda)$.
\begin{lemma}
    The corresponding eigenvector $\psi_{\lambda}(\bar{\mathbf{x}}) > 0 \ \forall \ \bar{\mathbf{x}} \in \Omega_r$.
\end{lemma}
Since the eigenvector $\psi_\lambda$ corresponding to this special value of $\lambda$ is guaranteed to be positive, we can conclude that $\lambda$ must be real-valued as well as unique.
\begin{lemma}
    There is a unique $\lambda_0 \in \Lambda \cap \mathbb{R}$ such that $r = 1$, and therefore $\lambda_0 \in \Sigma$.
\end{lemma}
With this lemma, we can indeed identify the unique value $\lambda_0$ for which $(I - \hat{K}(\lambda))^{-1}$ becomes singular, and if, as claimed in the following lemma, this is the dominant element of $\sigma$, then $\lambda_0$ determines the long term behavior of $B(t)$.
\begin{lemma}
    All other $\lambda \in \Sigma$ have $Re(\lambda) < \lambda_0$.
\end{lemma}
Finally, we see that $\lambda_0$ is the unique eigenvalue which makes the spectral radius of $\hat{K}(\lambda)$ equal to one, and all other eigenvalues in the set $\Sigma$ are separated from $\lambda_0$ by some positive horizontal distance.
 
Therefore, solving the abstract renewal equation reduces to solving the characteristic equation for $\hat{K}(\lambda)$, for which we must find the value of $\lambda$, and its corresponding eigenvector $\psi_\lambda$, such that $(\hat{K}(\lambda)\psi)(\bar{\mathbf{x}}) = 1 \psi(\bar{\mathbf{x}}) = \psi(\bar{\mathbf{x}})$. 

\subsubsection{Solution to the Eigenproblem as Applied to a Cell Population}
We seek a solution to the renewal equation wherein the contribution from the initial condition becomes negligible on a large time frame, and the behavior of $B(t, \bar{\mathbf{x}})$ can be described in terms of its dominant eigenvalue and corresponding eigenfunction, as in $$B(t, \bar{\mathbf{x}}) = K(B)(t, \bar{\mathbf{x}}) + \Phi(t, \bar{\mathbf{x}}) \sim e^{\lambda t} \psi(\bar{\mathbf{x}}),$$ as $t \to \infty$. 
We have shown that $\lambda$ and $\psi$ are the unique solutions to the characteristic equation for $\hat{K}(\lambda)$. That is, \begin{align} \begin{split}
    (K(t) e^{\lambda t}\psi)(\bar{\mathbf{x}}) &= e^{\lambda t}\psi(\bar{\mathbf{x}}) = \int_0^t \mathcal{K}(a, 2\bar{\mathbf{x}}) B(t - a, \bar{\mathbf{X}}(-a, 2\bar{\mathbf{x}}))da\\
    &= e^{\lambda t}\psi(\bar{\mathbf{x}}) = \int_0^t e^{\lambda (t - a)} \mathcal{K}(a, 2\bar{\mathbf{x}}) \psi(\bar{\mathbf{X}}(-a, 2\bar{\mathbf{x}}))da\\
    & \Rightarrow \psi(\bar{\mathbf{x}}) = \int_0^t e^{-\lambda a} \mathcal{K}(a, 2\bar{\mathbf{x}}) \psi(\bar{\mathbf{X}}(-a, 2\bar{\mathbf{x}}))da\\
    \psi(\bar{\mathbf{x}}) &= \int_0^t e^{-\lambda a} \Pi(a) J(a) \beta_1(a)  \psi(\bar{\mathbf{X}}(-a, 2\bar{\mathbf{x}}))da.
\end{split} \end{align}

Here we have arrived at the eigenequation for $\lambda$ and $\psi$. Completing the solution for the cyanobacterial population application reqires specification of a division condition and growth rate. For simplicity, we will assume exponential growth with constant growth rate $\alpha$ so that, $$\frac{dx_1}{dt} = \alpha t,$$ and the characteristic curve $$X_1(t, x_1) = x_1 e^{t \alpha}.$$ We further assume that carboxysomes age at the rate of time passing, as in $\displaystyle \frac{dx_{i \neq 1}}{dt} = 1$, and the characteristic curves $$X_{i \neq 1}(t, x_{i \neq 1}) = t + x_{i \neq 1}.$$ The Jacobian $J(t)$ then reduces to $$J(t) = e^{-t \alpha}.$$

As a condition for division, we will require a cell to double in length, so that if a cell is of size $\bar{\mathbf{x}}$ at birth, it will divide upon reaching size $2\bar{\mathbf{x}}$. Under this division condition, and with a constant growth rate, all cells will divide upon reaching the same age, $a_*$. Therefore, $\beta_1(a)$ becomes $\delta(a - a_*)$. 

We have, \begin{align}
    \begin{split}
        \psi(\bar{\mathbf{x}}) &= \int_0^t e^{-\lambda a} \Pi(a) J(a) \beta_1(a)  \psi(\bar{\mathbf{X}}(-a, 2\bar{\mathbf{x}}))da\\
        \psi(\bar{\mathbf{x}}) &= \int_0^t 2^{\frac{1}{\alpha} + 1} e^{-a(d + 1 + \alpha + \lambda)} \psi(\bar{\mathbf{X}}(-a, 2\bar{\mathbf{x}}))da\\
    &= \int_0^t \left( 2^{\frac{1}{\alpha}}e^{-a} \right) 2 e^{-a(d + \alpha + \lambda)} \psi(\bar{\mathbf{X}}(-a, 2\bar{\mathbf{x}}))da
    \end{split}
\end{align}

The term in parenthesis is the contribution to the survival probability due to cell division, \begin{align*}
    \exp \left[-\int_0^a \mu(\sigma, \bar{\mathbf{Y}}(\sigma - a, \bar{\mathbf{y}}))d\sigma \right] &= e^{-ad} \exp \left[-\int_0^a b(\sigma, \bar{\mathbf{Y}}(\sigma - a, \bar{\mathbf{y}}))d\sigma \right]\\
    &= e^{-ad}2^{\frac{1}{\alpha}}e^{-a},
\end{align*}
meaning, it gives the probability at each age that a cell of a given state will not divide. However, because we know that each cell must divide at the exact same age, $a_* = \frac{1}{\alpha}\ln{2}$, we can change this probability to be zero everywhere and 1 (or infinite) upon reaching age $a_*$. That is to say, the probability of cell division a large $t$ is entirely concentrated at $a_*$ and so we can represent it with the Dirac delta function $\delta(a - a_*)$. Making this change in the above equation allows us to evaluate the integral and find a solution for the eigenvalue $\lambda$ and it's corresponding eigenfunction $\psi$. \begin{align} \begin{split}
    \psi(\bar{\mathbf{x}}) &= \int_0^t \left( 2^{\frac{1}{\alpha}}e^{-a} \right) 2 e^{-a(d + \alpha + \lambda)} \psi(\bar{\mathbf{X}}(-a, 2\bar{\mathbf{x}}))da\\
    &= \int_0^t \delta(a - a_*) 2 e^{-a(d + \alpha + \lambda)} \psi(\bar{\mathbf{X}}(-a, 2\bar{\mathbf{x}}))da\\
    &= 2 e^{-a_*(d + \alpha + \lambda)} \psi(\bar{\mathbf{X}}(-a_*, 2\bar{\mathbf{x}}))
\end{split} \end{align}
The Jacobian term $e^{-a_* \alpha}$ must remain with the characteristic curve to avoid creating an imbalance as cells grow. The eigenfunction $\psi$ will take the form of a $\delta$ function as it will act to ``pick up" cells of the appropriate state, e.g., cells of length $2x_1$. However, since every cell divides upon doubling in size and reaching age $a_*$, every cell-state is a solution in that, the characteristic curve on the right-hand-side describes a cell currently of length $2x_1$ that was of length $x_1$ at birth, a time $a_*$ ago, exactly the cell required to produce the one on the left-hand-side. Therefore, $\psi(\bar{\mathbf{x}}) = C\delta(\bar{\mathbf{x}})$, where $C$ is a constant. Since deaths are included in this model, the arbitrary constant $C$ will absorb the constant probability of cell death. Finally, $\psi(\bar{\mathbf{x}}) = e^{-da_*}\delta(\bar{\mathbf{x}})$. Solving for $\lambda$ here is equivalent to setting $$2 e^{-a_*\lambda} = 1 \Rightarrow \lambda = \frac{\ln{2}}{a_*} = \alpha.$$

Therefore, after a short transient phase, we can characterize cell renewal for large values of $t$ as $$B(t, \bar{\mathbf{x}}) \sim e^{\lambda t}\psi(\bar{\mathbf{x}}) = e^{\alpha t} \left(e^{-a_*d}\delta(\bar{\mathbf{x}}) )\right .$$

\section{Concluding Remarks} \label{sec6}
Motivated by an example from cyanobacterial photosynthesis, we have proposed a high-dimensional, multi-structured model for the microbe population dynamics. Conventionally, it is understood that numerical simulations of PDE models become highly challenging for systems with the cardinality of independent variables greater than six.\footnote{There are of course examples where specific high dimensional PDEs can be solved using highly customized methods \cite{HanJentzenE2018ProcNatlAcadSciUSA,EHanJentzen2022Nonlinearity}} Accordingly, analytical approaches are a good strategy for investigating this class of models.  To the best of our knowledge, our efforts here are the first attempt to resolve the asymptotic behavior of a model with a large number of finitely many structuring variables.

From our analysis, we conclude that for the model \eqref{pde_model}, the solution has a short, transient phase, followed by asymptotic growth according to, $$n(t, a, \bar{\mathbf{x}}) \sim Ce^{\alpha t}\delta(\bar{\mathbf{x}})$$ as time grows toward infinity. The asymptotic solution says that, given an initial distribution, we will observe a wave-front like behavior wherein the initial cohort of cells will grow to double in size, and then reappear as cells of age zero at the exact same state that they and their mother cells had at birth. 

For this application, we chose a constant growth rate and simple division condition, and found that all cells divide at the same age. As a result, every new generation, though larger in number, has the same distribution with respect to the structure variables as every previous generation. Rather than reaching a structural distribution which remains stable for all time, the distribution of this population is only stable in the sense that it is repeated after every period of length $a_*$. There is some debate as to whether this should be interpreted as a periodic solution, as claimed for Bell's size-structured model in his 1967 conclusion \cite{bell1967cell}, or if this behavior constitutes convergence to a stable distribution at all, per Heijman's rebuttal \cite{MetzDiekmannLevin1986}. Nevertheless, we have demonstrated that an asymptotic solution wherein evolution in time is separate from a predictable state-space behavior does exist for the multi-structured model.

This solution lays a framework for future investigation of structured population dynamics. For example, it has been proposed that a more accurate condition for cell division, which leads to convergence in cell length at birth, is one in which cells grow by a constant amount before they divide, as opposed to reaching some set division length or doubling in size \cite{campos2014constant}. Under this assumption on division, and with a constant growth rate, we can similarly resolve the asymptotic behavior of our system by expressing the division age $a_*$, the time it takes to grow by a prescribed constant length $\Delta L$, as a function of cell length at birth, $X_1(-a, x_1)$. We have that $G(x)$ is the time it takes to grow from the smallest possible size, $x_m$ to size $x$, therefore $a_*$ would be $a_* = G(X_1(-a, x_1) + \Delta L) - a$.

In future work, we will apply this model, and the methods used to find an asymptotic solution, to investigate the effect of aging on carboxysome productivity by changing the growth rate of cells to be carboxysome age-dependent. Changing the model in this way, or expanding it to include competition between cells, will introduce further complexity and nonlinearity to our model equations. However, the work presented here will serve as the linearized model upon which analysis of the nonlinear extension would rely.

\backmatter

%
%
%

\bmhead{Acknowledgments}
We would like to thank Professors Nancy Rodriguez and Zack Kilpatrick (Department of Applied Mathematics, University of Colorado, Boulder) for insightful comments and suggestions on an earlier draft of this work, and Nicholas Hill (Department of Biochemisty, University of Colorado, Boulder) for helpful discussions in the development of this model.


\section*{Declarations}
\subsection*{Funding}
This work was supported in part by a 2017 Renewable and Sustainable Energy Institute Seed Grant and in part by National Science Foundation Grant Number 2054085. 
This research was also supported in part by an appointment with the NSF Mathematical Sciences Graduate Internship (MSGI) program sponsored by the National Science Foundation, Division of Mathematical Sciences (DMS). This program is administered by the Oak Ridge Institute for Science and Education (ORISE) through an interagency agreement between the U.S.
Department of Energy (DOE) and NSF. ORISE is managed by ORAU under DOE contract number DE-SC0014664. All opinions expressed in this paper are the author's and do not necessarily reflect the policies and views of NSF, ORAU/ORISE, or DOE.

\subsection*{Competing Interests}
The authors declare no competing interests.

\subsection*{Ethics approval}
Not Applicable
\subsection*{Consent to participate}
Not Applicable
\subsection*{Consent for publication}
Not Applicable
\subsection*{Availability of data and materials}
Not Applicable
\subsection*{Code availability}
\subsection*{Authors' contributions}
SLA, JCC, and DMB jointly developed the model and the general idea of the paper.  SLA wrote the text and the proofs. SLA and DMB edited the entire manuscript.

\begin{appendices}

\section{Proof of Jacobian and Exponential Term Equivalence} \label{appendix:jacobian}
After integrating the PDE (\ref{pde_model}) along characteristic curves, we are are left with the exponential term $\displaystyle \exp\left[-{\int_0^{\theta} \sum_{i=1}^k \frac{\partial v_i}{\partial X_i}d\sigma}\right],$ which we claim is equal to the Jacobian determinant, $$\displaystyle J(s) = \left \lvert \frac{\partial(A(\theta, a), \bar{\mathbf{x}}(\theta,\bar{\mathbf{x}}))}{\partial (a, \bar{\mathbf{x}})} \right \rvert.$$ The following is a proof of this claim using Liouville's Formula \cite{annosov1997ordinary}.
\begin{quote}
    \textbf{Liouville's Formula.}\\ Consider the first-order, linear, homogeneous ODE $$\frac{d\mathbf{r}}{dt} = K(t)\mathbf{r}(t), \ \ \mathbf{r}(0) = \mathbf{r}_0, \ \ \mathbf{r} \in \mathbb{R}^n$$ with fundamental matrix solution $\mathbf{\Phi}(t)$ satisfying $$\mathbf{\Phi}'(t) = K(t)\mathbf{\Phi}(t), \ \ \mathbf{\Phi}(0) = I.$$ The determinant of $\mathbf{\Phi}(x)$ then satisfies the ODE, $$\frac{d}{dt}\det(\mathbf{\Phi}(t)) = Tr(K(t))\det(\mathbf{\Phi}(t)),$$ with solution, $$\det(\mathbf{\Phi}(t)) = \det(\mathbf{\Phi}(t_0))\exp\left[\int_{t_0}^t Tr(K(s))ds \right]$$ where $Tr(K)$ is the trace of $K$, the sum of its diagonal elements.
\end{quote}

To make the extension to the arbitrary $k$-dimensional model clear, first assume that $k = 2$ and let $x_1 = x$, $x_2 = y$, $\frac{dx}{dt} = g(a, x, y)$, and $\frac{dy}{dt} = f(a, x, y)$, so that our population distribution is given by $n(t, a, x, y)$.

This gives the following system of ODEs and solutions,
\begin{equation*}
    \begin{cases} \frac{da}{dt} = 1, \hspace{1.5cm} A(t, a) = a - t\\ \frac{dx}{dt} = g(a, x, y), \ \ X(t, x) = G^{-1}(G(x) - t), \ \ G(x) = \int_{x_m}^x \frac{ds}{g(a, s, y)}\\ \frac{dy}{dt} = f(a, x, y), \ \ Y(t, y) = F^{-1}(F(y) - t), \ \ F(y) = \int_{0}^y \frac{ds}{f(a, x, s)}. \end{cases}
\end{equation*}

Capital letters again denote \textit{solutions} to this system of ODEs, i.e., characteristic curves. That is, $\mathbf{X} = (A, X, Y)^T$ are all functions of $t$, but $\mathbf{x} = (a, x, y)^T$ are not. Denote by $\mathbf{v}$ the vector field $(1, g, f)$, and lastly, note that $A(0, a) = a, X(0, x) = x$, and $Y(0, y) = y$.
    
The matrix $K(t)$ from Liouville's Formula is $D_\mathbf{X}\mathbf{v}(t, \mathbf{X})$: \begin{equation*}
        K(t) = \begin{bmatrix}\frac{\partial (1)}{\partial A} & \frac{\partial (1)}{\partial X} & \frac{\partial (1)}{\partial Y} \\ \frac{\partial g}{\partial A} & \frac{\partial g}{\partial X} & \frac{\partial g}{\partial Y} \\ \frac{\partial f}{\partial A} & \frac{\partial f}{\partial X} & \frac{\partial f}{\partial Y} \\\end{bmatrix}_{(A, X, Y)} = \begin{bmatrix}0 & 0 & 0 \\ \frac{\partial g}{\partial A} & \frac{\partial g}{\partial X} & \frac{\partial g}{\partial Y} \\ \frac{\partial f}{\partial A} & \frac{\partial f}{\partial X} & \frac{\partial f}{\partial Y} \\\end{bmatrix}_{(A, X, Y)} 
    \end{equation*}
and the trace of $K(t)$ is $\frac{\partial g}{\partial X} + \frac{\partial f}{\partial Y} = \nabla \cdot \mathbf{v}$, the divergence of $\mathbf{v}$.
    
The fundamental solution $\mathbf{\Phi}(t)$ for the system $\mathbf{\Phi}'(t) = K(t)\mathbf{\Phi}(t)$ is $D_{\mathbf{x}}\mathbf{X}(t, \mathbf{x})$, the Jacobian matrix of the characteristic curves: \begin{equation*}
        \mathbf{\Phi}(t) = \begin{bmatrix}\frac{\partial A}{\partial a} & \frac{\partial A}{\partial x} & \frac{\partial A}{\partial y} \\ \frac{\partial X}{\partial a} & \frac{\partial X}{\partial x} & \frac{\partial X}{\partial y} \\ \frac{\partial Y}{\partial a} & \frac{\partial Y}{\partial x} & \frac{\partial Y}{\partial y} \\\end{bmatrix} = \begin{bmatrix}1 & 0 & 0 \\ 0 & \frac{g(a, X, y)}{g(a, x, y)} & 0 \\ 0 & 0 & \frac{f(a, x, Y)}{f(a, x, y)} \\\end{bmatrix}
    \end{equation*}
which is equal to the identity matrix $I$ when evaluated at $t =0$.
    
To be clear, the off-diagonal entries such as, $\frac{\partial X}{\partial a}$, are, in fact, zero. When we write $G(x) = \int_{x_m}^x \frac{1}{g(a, s, y)}ds$, this is purely a function of $x$ as $a$ and $y$ are understood to be fixed. Therefore, differentiating $X$ with respect to $a$ or $y$ is zero, and not $g(a, X, y)\int_0^x - \frac{\frac{\partial g}{\partial a}(a, s, y)}{(g(a, s, y))^2}$, what you would get if $G(x)$ were $G(a, x, y)$.
    
The determinant of this matrix is, $\frac{g(a, X, y)}{g(a, x, y)} \cdot \frac{f(a, x, Y)}{f(a, x, y)}$, consistent with Liouville's Formula. 
\begin{align*}
        \frac{d}{dt}\det(\mathbf{\Phi}(t)) &= \frac{d}{dt}\left[\frac{g(a, X, y)}{g(a, x, y)} \cdot \frac{f(a, x, Y)}{f(a, x, y)}\right]\\ &= \Bigg[ \frac{-g(a, X, y)\frac{\partial g}{\partial x}(a, X, y)}{g(a, x, y)} \cdot \frac{f(a, x, Y)}{f(a, x, y)}\\ & \ \ - \frac{-f(a, x, Y)\frac{\partial f}{\partial y}(a, x, Y)}{f(a, x, y)} \cdot \frac{g(a, X, y)}{g(a, x, y)} \Bigg] \\ &= -\left(\frac{\partial g}{\partial x}(a, X, y) + \frac{\partial f}{\partial y}(a, x, Y)\right)\left(\frac{g(a, X, y)}{g(a, x, y)} \cdot \frac{f(a, x, Y)}{f(a, x, y)}\right)\\ &= -tr(K(t))\det(\mathbf{\Phi}(t))
    \end{align*}
Applying Liouville's formula, we see that $$\det(\mathbf{\Phi}(t)) = \det(\mathbf{\Phi}(0))\exp\left[-\int_{0}^t Tr(K(s))ds \right] = \exp\left[-\int_{0}^t \nabla \cdot \mathbf{v}(A, X, Y) ds \right].$$ The extension to the arbitrary $k$ case is tedious but follows straightforwardly.

Secondarily, we claim that $J(s)$ is bounded. Conditions imposed on the velocity functions $v_i$ such as, boundedness, continuity, and regularity, are what make $J$ bounded. On the interior of $\Omega$, $v_i(a, X, y)/v_i(a, x, y)$ is bounded. The product of bounded functions is also bounded, and as a result, $J(s)$ is always bounded, and $\lvert \lvert J(s)\lvert \lvert _\infty = 1$.

\section{Definitions for the Spectrum of Linear Operators} \label{appendix:definitions}
These definitions are collected from \cite{hunter2001applied} and \cite{arino1992some}.

Let $L$ be a closed and bounded linear operator with domain $\mathcal{D}(L)$ dense in a Banach space $U$. \begin{quote}
    \textbf{Definition 1}. The \textit{resolvent set} of $L$, denoted $\rho(L)$, is the open set $$\rho(L) = \{\lambda \in \mathbb{C} : (L - \lambda I) \text{ is one-to-one and onto} \}.$$ The Open Mapping Theorem \cite{hunter2001applied} implies that $(L - \lambda I)^{-1}$ is bounded for $\lambda \in \rho(L)$.
\end{quote}

\begin{quote}
    \textbf{Definition 2}. The \textit{resolvent operator}, denoted by $R(\lambda, L)$ or $R_\lambda$, is the operator-valued function $$(L - \lambda I)^{-1}$$ defined only on the set $\rho(L)$.
\end{quote}

\begin{quote}
    \textbf{Definition 3}. The \textit{spectrum} of $L$, denoted $\sigma(L)$, is the closed set $$\sigma(L) = \mathbb{C} \setminus \rho(L) =  \{\lambda \in \mathbb{C} : (L - \lambda I) \text{ is not boundedly invertible}\}.$$ The spectrum is composed of three disjoint sets: $$\sigma(L) = \sigma_P(L) \cup \sigma_C(L) \cup \sigma_R(L)$$ \begin{itemize}
        \item The \textit{point spectrum} $$\sigma_P(L) = \{\lambda \in \sigma(L) : (L - \lambda I) \text{ is not one-to-one}\}.$$ These are the eigenvalues of $L$.
        \item The \textit{continuous spectrum} \begin{align*}
            \sigma_C(L) = \{\lambda \in \sigma(L) &: (L - \lambda I) \text{ is one-to-one but not onto,} \\ &\text{ and the range of } (L - \lambda I) \text{ is dense in } B\}.
        \end{align*}
        \item The \textit{residual spectrum} \begin{align*}
            \sigma_R(L) = \{\lambda \in \sigma(L) &: (L - \lambda I) \text{ is one-to-one but not onto,} \\ &\text{ and the range of } (L - \lambda I) \text{ is not dense in } B\}.
        \end{align*}
    \end{itemize}{}  
\end{quote}

\begin{quote}
    \textbf{Definition 4}. The \textit{spectral radius}, denoted $r(L)$, is a bound on $\sigma(L)$---the radius of smallest disk centered at zero containing $\sigma(L)$, $$r(L) = \sup \{\lvert \lambda\lvert  : \lambda \in \sigma(L) \}.$$ Note that if $N$ is a nilpotent operator, i.e., $\exists m : N^m = 0$, $\sigma(N) = \{0\}$ and therefore, $r(N) = 0$. The spectral radius can be thought of as a measure of the distance from $L$ to the set of nilpotent operators, and this is reflected in the formula $$r(L) = \lim_{m \to \infty}\lvert \lvert L^m\lvert \lvert ^{1/m}.$$
\end{quote}

\begin{quote}
    \textbf{Definition 5}. The essential spectrum, denoted $\sigma_{ess}(L)$, is the set of $\lambda \in \sigma(L)$ such that at least one of the following holds: \begin{enumerate}
        \item The range of $(L - \lambda I)$ is not closed.
        \item The generalized eigenspace associated with $(L - \lambda I)$ is infinite dimensional.
        \item $\lambda$ is a limit point of $\sigma(L)$.
    \end{enumerate}
\end{quote}

\section{Proof of Semigroup Properties} \label{appendix:semgiroupprops}
$S(t)$ satisfies the four defining properties of a strongly continuous semigroup.
\subsubsection*{1. $S(t)$ is a continuous mapping.} 
\begin{proof}
    Let $\varepsilon > 0$ and $\lvert \lvert \phi(a, \bar{\mathbf{x}}) - \phi(b, \bar{\mathbf{y}})\lvert \lvert  < \varepsilon$ where $\lvert \lvert  \cdot \lvert \lvert $ is the $L^1$-norm. Recall that $\phi$ is continuously differentiable. Then, \begin{flalign*} \lvert \lvert &(S(t)\phi)(a, \bar{\mathbf{x}}) - (S(t)\phi)(b, \bar{\mathbf{y}})\lvert \lvert  \\ 
    &= \lvert \lvert \phi(a-t, \bar{\mathbf{G}}^{-1}(\bar{\mathbf{G}}(\bar{\mathbf{x}}) - t))\Pi(t, -t)J(t) - \phi(b-t, \bar{\mathbf{G}}^{-1}(\bar{\mathbf{G}}(\bar{\mathbf{y}}) - t))\Pi(t, -t)J(t)\lvert \lvert \\ 
    &= \lvert \lvert [\phi(a-t, \bar{\mathbf{G}}^{-1}(\bar{\mathbf{G}}(\bar{\mathbf{x}}) - t)) - \phi(b-t, \bar{\mathbf{G}}^{-1}(\bar{\mathbf{G}}(\bar{\mathbf{y}}) - t))]\Pi(t, -t)J(t)\lvert \lvert \\ 
    &\leq \lvert \lvert \phi(a-t, \bar{\mathbf{G}}^{-1}(\bar{\mathbf{G}}(\bar{\mathbf{x}}) - t)) - \phi(b-t, \bar{\mathbf{G}}^{-1}(\bar{\mathbf{G}}(\bar{\mathbf{y}}) - t))\lvert \lvert \cdot\lvert \lvert \Pi(t, -t)\lvert \lvert _{\infty} \cdot \lvert \lvert J(t)\lvert \lvert _\infty\\ &= \lvert \lvert J(t)\lvert \lvert _\infty\lvert \lvert \phi(a-t, \bar{\mathbf{G}}^{-1}(\bar{\mathbf{G}}(\bar{\mathbf{x}}) - t)) - \phi(b-t, \bar{\mathbf{G}}^{-1}(\bar{\mathbf{G}}(\bar{\mathbf{y}}) - t))\lvert \lvert  && \end{flalign*}
Consider only the difference in $\phi$ for now. We will show that this difference is bounded in the age-and-size-structured case where $\phi(a, \bar{\mathbf{x}})$ becomes $\phi(a, x)$, a function of age and size only, with the growth rate $g$ for $v_1$. The following arguments can be extended naturally to the multi-structured model. \begin{flalign*}
    \lvert \lvert \phi(&a-t, G^{-1}(G(x - t)) - \phi(b-t, G^{-1}(G(y) - t))\lvert \lvert \\
    &= \Bigg\lvert \Bigg\lvert \phi(a, x) - t\Bigg(\frac{1}{G'(x)}\frac{\partial \phi}{\partial x}(a, x) + \frac{\partial \phi}{\partial a}(a, x)\Bigg)\\ &- \Bigg[\phi(b, y) - t\Bigg(\frac{1}{G'(y)}\frac{\partial \phi}{\partial y}(b, y) + \frac{\partial \phi}{\partial b}(b, x)\Bigg)\Bigg]\Bigg\lvert \Bigg\lvert \\ 
    &= \Bigg\lvert \Bigg\lvert \phi(a, x) - t\Bigg(g(a, x)\frac{\partial \phi}{\partial x}(a, x) + \frac{\partial \phi}{\partial a}(a, x)\Bigg)\\ &- \Bigg[\phi(b, y) - t\Bigg(g(b, y)\frac{\partial \phi}{\partial y}(b, y) + \frac{\partial \phi}{\partial b}(b, y)\Bigg)\Bigg]\Bigg\lvert \Bigg\lvert \\ 
    &\leq \lvert \lvert \phi(a, x) - \phi(b, y)\lvert \lvert  + \Bigg\lvert \Bigg\lvert t\Bigg[g(b, y)\frac{\partial \phi}{\partial y}(b, y) + \frac{\partial \phi}{\partial b}(b, y)\\ &- \Bigg(g(a, x)\frac{\partial \phi}{\partial x}(a, x) + \frac{\partial \phi}{\partial a}(a, x)\Bigg)\Bigg]\Bigg\lvert \Bigg\lvert \\ 
    &\leq \varepsilon + \Bigg\lvert \Bigg\lvert t\Bigg(g(b, y)\frac{\partial \phi}{\partial y}(b, y) - g(a, x)\frac{\partial \phi}{\partial x}(a, x)\Bigg)\Bigg\lvert \Bigg\lvert \\ &+ \Bigg\lvert \Bigg\lvert t\Bigg(\frac{\partial \phi}{\partial b}(b, y) - \frac{\partial \phi}{\partial a}(a, x)\Bigg)\Bigg\lvert \Bigg\lvert  \\ 
    &= \varepsilon + \Bigg\lvert \Bigg\lvert t\Bigg(g(b, y)\frac{\partial \phi}{\partial y}(b, y) - g(a, x)\frac{\partial \phi}{\partial x}(a, x)\Bigg)\Bigg\lvert \Bigg\lvert  + \lvert t\lvert \varepsilon_3 \\ 
    &= \varepsilon + \Bigg\lvert \Bigg\lvert t\Bigg(g(b, y)\frac{\partial \phi}{\partial y}(b, y) - g(a, x)\frac{\partial \phi}{\partial x}(a, x)\Bigg)\Bigg\lvert \Bigg\lvert  + \lvert t\lvert \varepsilon_3 \\ 
    &= \varepsilon + \lvert t\lvert \Bigg\lvert \Bigg\lvert g(b, y)\frac{\partial \phi}{\partial y}(b, y) - g(b, y)\frac{\partial \phi}{\partial x}(a, x)\\ &+ g(b, y)\frac{\partial \phi}{\partial x}(a, x) - g(a, x)\frac{\partial \phi}{\partial x}(a, x)\Bigg\lvert \Bigg\lvert  + \lvert t\lvert \varepsilon_3\\ 
    &= \varepsilon + \lvert t\lvert \Bigg\lvert \Bigg\lvert g(b, y)\Bigg(\frac{\partial \phi}{\partial y}(b, y) - \frac{\partial \phi}{\partial x}(a, x)\Bigg)\\ &+ \frac{\partial \phi}{\partial x}(a, x)\Bigg(g(b, y) - g(a, x)\Bigg)\Bigg\lvert \Bigg\lvert  + \lvert t\lvert \varepsilon_3 && \end{flalign*}
    
    \begin{flalign*}
   &\leq \varepsilon + \lvert t\lvert \Bigg\lvert \Bigg\lvert g(b, y)\Bigg(\frac{\partial \phi}{\partial y}(b, y) - \frac{\partial \phi}{\partial x}(a, x)\Bigg)\Bigg\lvert \Bigg\lvert \\ &+ \lvert t\lvert \Bigg\lvert \Bigg\lvert \frac{\partial \phi}{\partial x}(a, x)\Bigg(g(b, y) - g(a, x)\Bigg)\Bigg\lvert \Bigg\lvert  + \lvert t\lvert \varepsilon_3\\ 
   &\leq \varepsilon + \lvert t\lvert \Bigg(\lvert \lvert g(b, y)\lvert \lvert _\infty \Bigg\lvert \Bigg\lvert \frac{\partial \phi}{\partial y}(b, y) - \frac{\partial \phi}{\partial x}(a, x)\Bigg\lvert \Bigg\lvert \\ &+ \Bigg\lvert \Bigg\lvert \frac{\partial \phi}{\partial x}(a, x)\Bigg\lvert \Bigg\lvert _\infty \Bigg\lvert \Bigg\lvert g(b, y) - g(a, x)\Bigg\lvert \Bigg\lvert \Bigg) + \lvert t\lvert \varepsilon_3\\ 
    &\leq \varepsilon + \lvert t\lvert \Bigg(\varepsilon_a\lvert \lvert g(b, y)\lvert \lvert _\infty + \varepsilon_b\Bigg\lvert \Bigg\lvert \frac{\partial \phi}{\partial x}(a, x)\Bigg\lvert \Bigg\lvert _\infty \Bigg) + \lvert t\lvert \varepsilon_3\\ 
    &= \varepsilon + \lvert t\lvert (\varepsilon_2 + \varepsilon_3) && \end{flalign*} where $\varepsilon_a = \varepsilon_2/2\lvert \lvert g\lvert \lvert _\infty$, $\varepsilon_b = \varepsilon_2/2\lvert \lvert \partial \phi / \partial x\lvert \lvert _\infty$, and $\varepsilon_3$ comes from the fact that $\phi$ must be continuously differentiable. (You could also get $\varepsilon_2$ from the fact that the product of continuous functions are is continuous.)\\
    Finally, $$\lvert \lvert (S(t)\phi)(a, x) - (S(t)\phi)(b, y)\lvert \lvert  \leq \varepsilon + \lvert t\lvert (\varepsilon_2 + \varepsilon_3) = \delta(\varepsilon).$$ Since for every $\varepsilon$ such that $\lvert \lvert \phi(a, x) - \phi(b, y)\lvert \lvert  < \varepsilon$ there exists a $\delta(\varepsilon) > 0$ such that $\lvert \lvert (S(t)\phi)(a, x) - (S(t)\phi)(b, y)\lvert \lvert  \leq \delta$, the mapping is continuous.
\end{proof} 
    
\subsubsection*{2. $S(0) = I$.} \begin{proof}
        \begin{align*}
        (S(0)\phi)(a, \bar{\mathbf{x}}) &= \phi(a-0, \bar{\mathbf{G}}^{-1}(\bar{\mathbf{G}}(\bar{\mathbf{x}}) - 0))J(0)\\ 
        &\times \exp\left[-\int_0^{0} \mu(a - t', \bar{\mathbf{G}}^{-1}(\bar{\mathbf{G}}(\bar{\mathbf{x}}) - t'))dt' \right]\\ &= \phi(a, \bar{\mathbf{x}}) = (I\phi)(\bar{\mathbf{x}})\\ &\Rightarrow S(0) = I
    \end{align*}
    \end{proof} 
    
\subsubsection*{3. The semigroup property: $S(s)S(t)\phi = S(t + s)\phi$.} \begin{proof}
        \begin{align*}
        S(s)&[(S(t)\phi)](a, x) = (S(s)\phi)(a-t, \bar{\mathbf{G}}^{-1}(\bar{\mathbf{G}}(\bar{\mathbf{x}}) - t))\Pi(t, -t)J(t)\\ &= \phi((a-t)-s, \bar{\mathbf{G}}^{-1}((\bar{\mathbf{G}}(\bar{\mathbf{x}}) - t)-s))\Pi(t, -t)\\ &\times e^{\left[-\int_t^{s} \mu(a - t', \bar{\mathbf{G}}^{-1}(\bar{\mathbf{G}}(\bar{\mathbf{x}}) - t'))dt'\right]}J(t)e^{\int_t^s \nabla \cdot \bar{\mathbf{v}} dt'}\\ &= \phi(a-(t+s), \bar{\mathbf{G}}^{-1}(\bar{\mathbf{G}}(\bar{\mathbf{x}}) - (t+s))\Pi(t+s, -t)J(t+s)\\ &= (S(t+s)\phi)(a, \bar{\mathbf{x}})
    \end{align*}
    Where \begin{align*}
        \Pi(t, -t) &e^{\left[-\int_t^{s} \mu(a - t', \bar{\mathbf{G}}^{-1}(\bar{\mathbf{G}}(\bar{\mathbf{x}}) - t'))dt'\right]} = e^{\left[-\int_0^{t} \mu(a - t', \bar{\mathbf{G}}^{-1}(\bar{\mathbf{G}}(\bar{\mathbf{x}}) - t'))dt'\right]}\\ &\times e^{\left[-\int_t^{s} \mu(a - s', \bar{\mathbf{G}}^{-1}(\bar{\mathbf{G}}(\bar{\mathbf{x}}) - s'))ds'\right]}\\ &= \exp\left[-\left(\int_0^{t} \mu dt' + \int_t^{s} \mu dt'\right)\right]\\ &= \exp\left[-\int_0^{t+s} \mu(a - t', \bar{\mathbf{G}}^{-1}(\bar{\mathbf{G}}(\bar{\mathbf{x}}) - t')) dt'\right]\\ &= \Pi(t+s, -t) 
    \end{align*}
    and similarly for $J(t+s).$
\end{proof}
    
\subsubsection*{4. Strong continuity.} 
$\lim_{t \searrow 0} \lvert \lvert S(t)\phi - \phi\lvert \lvert  = 0$ where $\lvert \lvert  \cdot \lvert \lvert $ is the operator norm in the Banach space $U = L^1$. \begin{proof}
        \begin{align*}
        \lim_{t \searrow 0} \lvert \lvert S(t)\phi - \phi\lvert \lvert  &= \lim_{t \searrow 0} \lvert \lvert \phi(a-t, \bar{\mathbf{G}}^{-1}(\bar{\mathbf{G}}(\bar{\mathbf{x}}) - t))\Pi(t, -t)J(t) - \phi(a, x)\lvert \lvert \\ &= \lvert \lvert \lim_{t \searrow 0}\phi(a-t, \bar{\mathbf{G}}^{-1}(\bar{\mathbf{G}}(\bar{\mathbf{x}}) - t))\Pi(t, -t)J(t) - \phi(a, \bar{\mathbf{x}})\lvert \lvert \\ &= 0 
    \end{align*}
    Therefore, $\{S(t)\}_{t \geq 0}$ forms a strongly continuous semigroup on $L^1(\mathbb{R}_+ \times \Omega)$.
\end{proof} 

\section{Infinitesimal Generator Proof}
\label{appendix:generatorproof}

\begin{lemma}
    $\mathcal{A}$, as defined in Equation \eqref{aphi}, is the infinitesimal generator of the strongly-continuous semigroup $S(t)$.
\end{lemma}

\begin{proof}
    An operator $\mathcal{A}$ is the generator of a semigroup if $$\mathcal{A}\phi = \lim_{t \searrow 0}\frac{1}{t}(S(t)\phi - \phi)$$ for every $\phi \in \mathcal{D}(\mathcal{A}).$
    \begin{align}
        \begin{split}
            \lim_{t \searrow 0}\frac{1}{t}\left(S(t)\phi - \phi\right) &= \lim_{t \searrow 0}\frac{1}{t}\left(\phi(a-t, \bar{\mathbf{G}}^{-1}(\bar{\mathbf{G}}(\bar{\mathbf{x}}) - t))\Pi(t, -t)J(t) - \phi(a, \bar{\mathbf{x}})\right)\\
            =  &\lim_{t \searrow 0}\frac{1}{t}\Big(\phi(a-t, \bar{\mathbf{G}}^{-1}(\bar{\mathbf{G}}(\bar{\mathbf{x}}) - t))\exp\left[-\int_0^{t} \mu(a - t', \bar{\mathbf{G}}^{-1}(\bar{\mathbf{G}}(\bar{\mathbf{x}}) - t'))dt' \right]\\ &\cdot \exp\left[-\int_0^t \nabla \cdot \bar{\mathbf{v}}(a - t, \bar{\mathbf{G}}^{-1}(\bar{\mathbf{G}}(\bar{\mathbf{x}}) - t'))dt'\right] - \phi(a, \bar{\mathbf{x}})\Big)\\ \end{split} \end{align}
            
In the next step, we proceed by expanding each term in a Taylor series about zero from the right. \begin{align}
            \begin{split}
            = &\lim_{t \searrow 0}\frac{1}{t}\Big(\phi(a, \bar{\mathbf{x}}) - t\Big[\frac{\partial \phi}{\partial a}(a, \bar{\mathbf{x}}) + \sum_{i=1}^k v_i(a, \bar{\mathbf{x}})\frac{\partial \phi}{\partial x_i}(a, \bar{\mathbf{x}}) + \phi(a, \bar{\mathbf{x}})\sum_{i=1}^k \frac{\partial v_i}{\partial x_i}(a, \bar{\mathbf{x}})\\ &+ \mu(a, \bar{\mathbf{x}})\phi(a, \bar{\mathbf{x}})\Big] + \mathcal{O}(t^2) - \phi(a, \bar{\mathbf{x}})\Big)\\
            = &\lim_{t \searrow 0}\Big(\frac{\partial \phi}{\partial a}(a, \bar{\mathbf{x}}) + \sum_{i=1}^k v_i(a, \bar{\mathbf{x}})\frac{\partial \phi}{\partial x_i}(a, \bar{\mathbf{x}}) + \phi(a, \bar{\mathbf{x}})\sum_{i=1}^k \frac{\partial v_i}{\partial x_i}(a, \bar{\mathbf{x}})\\ &+ \mu(a, \bar{\mathbf{x}})\phi(a, \bar{\mathbf{x}}) +\mathcal{O}(t)\Big)\\
            = &-\Big(\frac{\partial \phi}{\partial a}(a, \bar{\mathbf{x}}) + \sum_{i=1}^k v_i(a, \bar{\mathbf{x}})\frac{\partial \phi}{\partial x_i}(a, \bar{\mathbf{x}}) + \phi(a, \bar{\mathbf{x}})\sum_{i=1}^k \frac{\partial v_i}{\partial x_i}(a, \bar{\mathbf{x}})\\ &+ \mu(a, \bar{\mathbf{x}})\phi(a, \bar{\mathbf{x}})\Big)\\
            =& -\left(\frac{\partial \phi}{\partial a} + \sum_{i=1}^k v_i\frac{\partial \phi}{\partial x_i} + \phi\sum_{i=1}^k \frac{\partial v_i}{\partial x_i} + \mu \phi \right)\\
            =& \mathcal{A}\phi
        \end{split}
    \end{align}
Therefore, $\mathcal{A}$, as defined in Equation \eqref{aphi}, is the infinitesimal generator of the strongly continuous one-parameter semigroup $\{S(t)\}_{t \geq 0}$.
\end{proof}

\end{appendices}



\bibliographystyle{bst/sn-mathphys.bst}
\bibliography{heybib}
\end{document}